\documentclass[journal,onecolumn]{IEEEtran}

\AtBeginDocument{%
  \providecommand\BibTeX{{%
    \normalfont B\kern-0.5em{\scshape i\kern-0.25em b}\kern-0.8em\TeX}}}

\def\P{\mathbb{P}}

\def\et{\textit{et al.}}

\newenvironment{proof}{Proof:}{\hfill$\square$}

	\newcommand{\Dcal}   {{\mathcal D }}

	\newcommand{\Kcal}   {{\mathcal K }}

	\newcommand{\Vcal}   {{\mathcal V }}

\def\R{{\mathbb{R}}}
\newcommand{\Rbar}{\bar{\mathbb {R}}}

\def\N{{\mathbb{N}}}
\def\Z{{\mathbb{Z}}}

\def\E{{\mathbb{E}}}

	\newcommand{\supp}{{\operatorname {supp}\,}}
	\newcommand{\ess}{{\operatorname {ess}}}

\usepackage{amsmath}
\usepackage{enumitem}
\usepackage{relsize}
\newtheorem{theorem}{Theorem}

\newtheorem{lemma}[theorem]{Lemma}
\newtheorem{corollary}[theorem]{Corollary}

\usepackage{colordvi}
\usepackage{amsfonts}
\usepackage{graphicx}
\usepackage{epsfig}
\usepackage{subcaption}
\usepackage{color}
\usepackage{mathtools, cuted}
\usepackage{url}
\usepackage{wrapfig}
\usepackage{tikz}
\usepackage{enumitem}

\newcommand{\1}{{\mathchoice {1\mskip-4mu\mathrm l}      % Blackboard bold 1
{1\mskip-4mu\mathrm l}
{1\mskip-4.5mu\mathrm l} {1\mskip-5mu\mathrm l}}}
\makeatletter
\renewcommand*{\p@section}{\S\,}
\renewcommand*{\p@subsection}{\S\,}
\renewcommand*{\p@subsubsection}{\S\,}
\makeatother

\usetikzlibrary{arrows,matrix,backgrounds,fit,calc,automata,shapes,positioning,decorations.pathreplacing}
\usetikzlibrary{calc}

\tikzset{
    between/.style args={#1 and #2}{
         at = ($(#1)!0.5!(#2)$)
    }
}

\begin{document}

\title{Poly-Exp Bounds in Tandem Queues}

%\author{Sima Mehri and Florin Ciucu}
%\affiliation{\institution{University of Warwick}}

\author{\IEEEauthorblockN{Florin Ciucu and Sima Mehri}\\\IEEEauthorblockA{University of Warwick}}

\maketitle

\begin{abstract}
When the arrival processes are Poisson, queueing networks are well-understood in terms of the product-form structure of the number of jobs $N_i$ at the individual queues; much less is known about the waiting time $W$ across the whole network. In turn, for non-Poisson arrivals, little is known about either $N_i$'s or $W$.

This paper considers a tandem network
$$GI/G/1\rightarrow \cdot/G/1\rightarrow\dots\rightarrow\cdot/G/1$$
with general arrivals and light-tailed service times. The main result is that the tail $\P(W>x)$ has a polynomial-exponential (Poly-Exp) structure by constructing upper bounds of the form
$$(a_{I}x^{I}+\dots+a_1x+a_0)e^{-\theta x}~.$$
The degree $I$ of the polynomial depends on the number of bottleneck queues, their positions in the tandem, and also on the `light-tailedness' of the service times. The bounds hold in non-asymptotic regimes (i.e., for \textit{finite} $x$), are shown to be sharp, and improve upon alternative results based on large deviations by (many) orders of magnitude. The overall technique is also particularly robust as it immediately extends, for instance, to non-renewal arrivals.
\end{abstract}

\section{Introduction}
A landmark result in queueing theory is the product-form structure of the number of jobs $N_i$ at the individual queues in steady-state, i.e.,
$$\P\left(N_1=n_1,\dots,N_M=n_M\right)= \prod_i\P\left(N_i=n_i\right)~.$$
This property was first proved in Jackson networks where external arrivals are Poisson processes, service times are exponentially distributed, the scheduling at each queue is FIFO, and the jobs are probabilistically routed through the network. Several major generalizations (e.g., BCMP or Kelly networks) allow for instance for more general service time distributions or other scheduling algorithms; such results are covered in most books on queueing theory.

What is remarkable about the product-form structure is that queues seemingly behave as if they were statistically independent, even if generally are not; in fact, even in the case of a Jackson network, queues are not independent when jobs are routed back to some of the already traversed queues. The key benefit of the product-form result is that it immediately lends itself to the distributions of the marginal queue sizes $N_i$ by taking the limit $n_j\rightarrow\infty$ for all $j\neq i$.

While each queue is therefore well-understood in isolation in product-form networks, much less in known about end-to-end waiting times. For instance, even in the $M/M/1\rightarrow M/M/1$ case, unlike the sojourn times at the two queues which are independent (Reich~\cite{Reich63}), the corresponding waiting times are not (Burke~\cite{Burke64}). The joint distribution of waiting times is known in the $M/M/1\rightarrow M/M/1$ case (Karpelevitch and Kreinin~\cite{Karpelevitch92}), as well as the distribution of the end-to-end waiting time in a more general case when the second queue is served by multiple servers (Kr\"{a}mer~\cite{Kraemer73}). These results were generalized for a large class of networks with Poisson arrivals by Ayhan and Baccelli~\cite{Ayhan01}, but only in terms of providing the joint Laplace-Stieltjes transform (LST) which generally requires numerical inversions; explicit expression for the distribution of individual waiting times in such networks were later obtained by Ayhan and Seo~\cite{Ayhan02}.

Similar results on sojourn times exist with a notable exception. A very general result is a product-form expression for the joint LST over a so-called quasi overtake-free path in both open and closed product-form networks (e.g., BCMP). This is stated in the survey of Boxma and Daduna~\cite{Boxma90} (see Theorem 2.4) and generalizes prior results for open or closed networks; the same survey additionally addresses the issue of numerical computations of inverting the LST, and also several approximation techniques in both product and non-product form networks. For a more focused survey on numerical computations of sojourn times' quantiles see Harrison and Knottenbelt~\cite{Harrison04}. As mentioned earlier, unlike waiting times in a $M/M/1\rightarrow M/M/1$ tandem, sojourn times are independent. This exceptional property immediately extends to feedforward/tree Jackson networks, and as an immediate consequence the end-to-end sojourn time has an Erlang distribution (as the sum of independent exponential random variables). Walrand and Varaiya \cite{Walrand80} generalized this result to open Jackson networks subject to non-overtaking paths.

All previous results, along with many related others (see also the comprehensive survey by Disney and K\"{o}nig~\cite{Disney85}) have two common features: they are exact but are restricted to Poisson arrivals. One effective approach to relax the Poisson assumption, and yet remain in the realm of exact results, is large deviations. Ganesh~\cite{Ganesh98} studies the sojourn times $D$ in tandem queues under essentially the same general conditions as in this paper; the main result is that $D$ decays exponentially
$$\P\left(D> x\right)\approx e^{-\theta x}~,$$
for some asymptotic decay rate $\theta>0$; more formally $\lim_{x\rightarrow\infty}\frac{\log \P\left(D> x\right)}{x}=-\theta$, i.e., the result is \textit{asymptotically} exact.

An alternative approach is network calculus ~\cite{HLiu03,Book-Chang,Fidler06}. While intrinsically also based on large deviations, a key feature of network calculus is to effectively collect all terms preceding the exponential rather than discarding them by invoking the limit $\lim_{x\rightarrow\infty}\frac{\log \P\left(D> x\right)}{x}$. In this way, non-asymptotic bounds on the tail of $D$ (and also $W$) follow in a more or less straightforward manner; some explicit derivations are later provided in~\ref{sec:nc}. While such results could be obtained for broad classes of arrival processes, including non-renewal or heavy-tailed processes (e.g.,~\cite{Fidler06,LiBuCi12}), their proverbial drawback is the poor numerical accuracy.

The goal of this paper is to provide a sharp and non-asymptotic analysis for the end-to-end waiting time $W$ in queueing networks with non-Poisson arrivals. We consider in particular the $GI/G/1\rightarrow \cdot/G/1\rightarrow\dots\rightarrow\cdot/G/1$ tandem network with $M$ queues, general arrivals, and light-tailed service times (i.e., having some finite moment generating functions). Besides the non-Poisson nature of the input, what makes the analysis of $\P(W>x)$ exceptionally difficult is the structure of $W$ in terms of $M$ nested random walks.

At a very high level, our approach follows the standard $GI/G/1$ analysis of formulating and solving an integral/renewal equation. The crucial difference is that rather than using the very distribution $\P(W\leq x)$ as the integrand, we first decompose $W$ into two `suitable' random walks and use their joint distribution as the integrand. The obtained renewal equation enables the existence of polynomial-exponential (poly-exp) type of solutions, which immediately lend themselves to poly-exp bounds on $\P(W>x)$. For illustration purposes we provide explicit constructions for the $GI/M/1\rightarrow \cdot/M/1$ tandem; while the underlying computations involve elementary integration the overall analysis is tedious. Numerical results highlight that the bounds are not only sharp, but also that their poly-exp structure is instrumental in capturing the concave behavior of $\P(W>x)$, on the log-log scale, at small $x$.

In the following we first analyze a tandem of two queues -- to convey the essential ideas in a notation-light setting. The general case of $M$ queues will follow in~\ref{sec:gc}; while the key ideas carry over from $M=2$, some technical generalizations are not straightforward for which reason we provide the full proofs. Next we explicitly analyze $GI/M/1\rightarrow \cdot/M/1$ and show numerical comparisons against both simulations and alternative bounds based on large deviations. Lastly we discuss several immediate extensions and open problems, and summarize the paper. An Appendix contains auxiliary technical results and several remaining proofs.

\section{Simplest Case: Two Queues ($M=2$)}\label{sec:m2}
To ease the understanding of our overall approach we start with the case of two queues in tandem. The general and expectedly more tedious case of $M$ queues will be covered in~\ref{sec:gc}.

	\begin{figure}[h]
		\centering
		\includegraphics[width=0.45\linewidth]{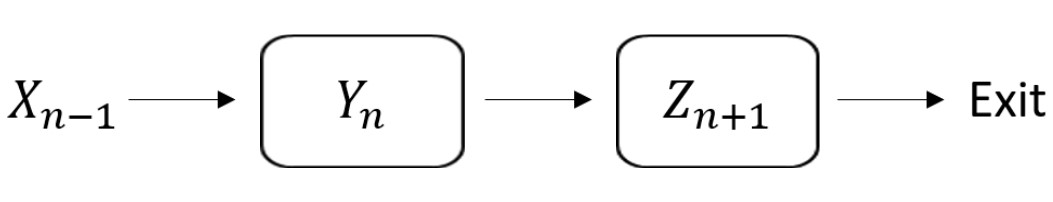}
		\caption{Inter-arrival and service times for job $n$ in a tandem of two queues}
		\label{fig:2queues}
	\end{figure}

We consider the tandem $GI/G/1\rightarrow\cdot/G/1$. Jobs $0,1,2,\dots$ traverse the two queues, each with infinite capacity; job $0$ arrives at time $0$ and job $n$ arrives at time $X_0+\cdots+X_{n-1}$ \footnote{By convention, all sums from the paper where the last index precedes the first are zero, e.g., $X_0+X_{-1}=0$.}. The service times of job $n$ at the first and second queues are $Y_n$ and $Z_{n+1}$, respectively; see Fig.~\ref{fig:2queues}. All sequences $(X_n)$, $(Y_n)$, and $(Z_n)$ are i.i.d. and mutually independent. We also assume the stability condition $\E[X_0]>\max\{\E[Y_1],\E[Z_2]\}$.

The service times are assumed to be \textit{light-tailed}, i.e., they admit finite moment generating functions. Concretely, for the first queue, we assume that $\theta^+:=\sup\{\theta>0:\E[e^{\theta (Y-X)}]< \infty\}\in(0,\infty]$. If $\E[e^{\theta^+ (Y-X)}]\geq 1$, which covers most cases of interest, then $\E[e^{\theta (Y-X)}]=1$ admits a unique solution which dictates the asymptotic decay rates of waiting/sojourn times (\cite{Kingman64}). Otherwise, if $\E[e^{\theta^+ (Y-X)}]<1$, then the asymptotic decay rate would be $\theta^+$; for an example of such a `very-light' distribution see Appendix~\ref{app:vld}.

The main metric of interest is the sum $W$ of the waiting times of \textit{some} job across the tandem.

\subsection{Sums of Waiting Times vs Product-Form Results}
The fundamental difficulty in treating sums of waiting times, as opposed to product-form results, can be understood from Burke's proof that the former are not independent~\cite{Burke64}; a similar argument was informally made by Harrison and Knottenbelt~\cite{Harrison04}.

Consider the $M/M/1\rightarrow M/M/1$ special case in steady-state and denote by $N_1(t)$ and $N_2(t)$ the number of jobs in the two queues at time $t$, and by $W_1$ and $W_2$ the waiting times of \textit{some} job. Burke showed that
$$\P(W_2=0\mid W_1=0)>\P(W_2=0)$$
by explicitly computing the left term; the right is simply $\P(N_2(t)=0)$ by using the Arrival Theorem for Jackson networks (see, e.g., Walrand~\cite{Walrand89}, p. 73).

Denoting by $T$ and $S$ the arrival and service times of a job experiencing zero waiting time in the first queue, the previous relation becomes
$$\P\left(N_2(T+S)=0\mid N_1(T)=0\right)>\P(N_2(T+S)=0)~,$$
i.e., $N_1(T)$ and $N_2(T+S)$ are not independent, even if $T$ is a stopping time independent of $S$.

It should now be apparent that a key difficulty in jointly dealing with $W_1$ and $W_2$ is the equivalence with jointly and implicitly dealing with $N_1$ and $N_2$ at \textit{different} (random) times where independence lacks. This is unlike the product-form result dealing with $N_1$ and $N_2$ at the \textit{same} time
$$\P(N_1(t)=n_1,N_2(t)=n_2)=\P(N_1(t)=n_1)\P(N_2(t)=n_2)~.$$

\subsection{A novel representation of sums of waiting times}
We proceed by first computing the exit time of job $n$ from the second queue (i.e., from the whole tandem)
	\[\max_{0\leq i< j\leq n+1}X_0+\cdots+X_{i-1}+Y_{i}+\cdots+Y_{j-1}+Z_{j}+\cdots+Z_{n+1}~,\]
which follows by a typical argument involving Lindley's recursions.

For convenience of future notation we change all indexes from $k$ to $n+2-k$; in particular, the service times of job $n$ become $Y_2$ and $Z_1$, the inter-arrival time between jobs $n$ and $n-1$ becomes $X_3$, the service times of job $n-1$ become $Y_3$ and $Z_2$, and so on. The exit time of job $n$ can then be rewritten as
	\[\max_{1\leq i<j\leq n+2}X_{n+2}+\cdots+X_{j+1}+Y_{j}+\cdots+Y_{i+1}+Z_{i}+\cdots+Z_{1}~.\]

Therefore, the waiting time of job $n$ in the tandem has the same distribution as
	\begin{align}W&=\max_{1\leq i<j\leq n+2}Z_1+\cdots+Z_{i}+Y_{i+1}+\cdots +Y_{j}+X_{j+1}+\cdots +X_{n+2} - \left(Z_1+Y_2+X_3+\cdots +X_{n+2}\right)\notag\\&=\max\left\{0,T^1+\left(Z_2-Y_2\right)_+, T^2+Z_2-Y_2\right\}~.\label{eq:W2nodes}\end{align}
We denote $u_+:=\max\left\{u,0\right\}$ and
	\[ T^1:=\max_{3\leq  i\leq n+2}U_3+\cdots+U_{i}\]
	\[ T^2:=\max_{3\leq  i<j\leq n+2}V_3+\cdots+V_{i}+U_{i+1}+\cdots+U_{j}~,\]
		where $V\simeq  Z-X$ and $U\simeq Y-X$ (`$\simeq$' stands for equality in distribution). Note that $T^2$ is undefined when $n=1$, and both $T^1$ and $T^2$ are undefined when $n=0$; in these cases, the corresponding undefined sums from the expression of $W$ are set to 0, e.g., $W=0$ when $n=0$.

The sojourn time (delay) of job $n$, i.e., the sum of the waiting and the two service times
$$D:= W+Z_1+Y_2$$
recovers in the limit $n\rightarrow\infty$ the standard representation (see Ganesh~\cite{Ganesh98} or Foss~\cite{Foss07})
$$\sum_{0\leq i\leq j}Z_0+\dots+Z_i+Y_i+\dots+Y_j-(X_0+\dots+X_{j-1})~,$$
by slightly re-indexing $(X_n)$, $(Y_n)$, and $(Z_n)$. Unlike this more compact representation, the key difference in our representation from (\ref{eq:W2nodes}) stands in avoiding the indexes' overlap `$i\leq j$' in the sum; we do so by separating $T^1$ from $T^2$, the latter involving a sum with `$i<j$'. As detailed shortly in~\ref{sec:ierw}, this separation is crucial in own overall approach.

In the limit $n\to\infty$ there exist unique stationary distributions for $W$ and $D$ (Loynes~\cite{Loynes64}); we will mainly address $W$ and make some occasional remarks about $D$.

\subsection{Generic Construction of Upper and Lower Bounds}
The representation of $W$ from (\ref{eq:W2nodes}) is an essential step in our overall approach. Its usefulness will become apparent after the next theorem which enables the construction of generic upper and lower bounds on $\P(W>x)$. First some additional notation: $u\wedge v:=\min\{u,v\}$ and $u\vee v:=\max\{u,v\}$.
\begin{theorem}\label{th:psigammaeta}
		Let $U$ and $V$ be  two random variables  satisfying $\P(U>0)>0$ and $\P(V>0)>0$.
		\begin{enumerate}[label=(\alph*)]
		\item The integral equation
\begin{equation}\label{equationpsi}
			\E\left[\1_{\{u\geq  U\}}\psi\left((u-U)\wedge (v-V), v-V \right)\right]=\psi(u,v)
		\end{equation}
admits a unique solution in the class of bounded functions on $\Rbar^2$  having the limit $\psi(\infty,\infty)=\lim_{u,v\to \infty}\psi(u,v)=1$. This is given by
$$\psi(u,v):=\P( T^1_{1}\leq u, T^{2}_{1}\leq v )~,$$ where
	 \[ T^1_k:=\sup_{k\leq  i<\infty}U_k+\cdots+U_{i}\]
		\[ T^2_k:=\sup_{k\leq  i<j<\infty}V_k+\cdots+V_{i}+U_{i+1}+\cdots+U_{j}\]
for $k\geq1$ and $(U_1, V_1)$, $(U_2, V_2),   \ldots $ being i.i.d. copies of $(U, V)$; by notation, $\Rbar:=\R\cup\{\pm\infty \}$.
	\item	Assume that the function $\gamma:\Rbar^2\to (-\infty,K_\gamma]$, for some finite $K_\gamma$, satisfies
	 for all  $(u,v)\in \supp(\gamma\vee 0)\subseteq \Rbar^2$
		\begin{equation}\label{pre-estimatepsi}
			\E\left[\1_{\{u\geq  U\}}\gamma\left((u-U)\wedge (v-V), v-V \right)\right]\geq \gamma(u,v)~.
		\end{equation}
	If $\gamma(\infty, \infty):=\limsup_{u,v\to \infty}\gamma(u,v)=1$  then $\psi\geq\gamma$.
		\item	Assume that the  function $\eta:\Rbar^2\to [0,\infty)$ satisfies
	for all  $(u,v)\in \supp(\psi)$ with $v\geq u$
	\begin{equation}\label{pre-estimateeta}
		\E\left[\1_{\{u\geq  U\}}\eta\left((u-U)\wedge (v-V), v-V\right)\right]\leq \eta(u,v)~.
	\end{equation}
	If
		$\eta(\infty, \infty):=\liminf_{u,v\to \infty}\eta(u,v)=1$ then $\psi(u,v)\leq\eta(u,v)$ for all $v\geq u$.
			\end{enumerate}
	\end{theorem}

As shown shortly, the problem of finding upper and lower bounds on $\P(W>x)$ reduces to the problem of finding the functions $\gamma$ and $\eta$ in (b) and (c), respectively. As a technical remark, the additional restriction $v\geq u$ from (c) could strengthen the lower bounds by imposing a weaker condition on $\eta$ in \eqref{pre-estimateeta}.

	\begin{proof}
		For part (a) we first prove that the given $\psi$ satisfies \eqref{equationpsi}; uniqueness will follow after proving (b).  We have
		\begin{align*}
			&\psi(u,v)=\P(T^{1}_{1}\leq u, T^{2}_{1}\leq v  )
			\\&=\P\left( T^{1}_2\leq  (u-U_1)\wedge (v-V_1), T^{2}_{2}\leq v-V_1   , U_1\leq u\right),\\&=\P\left( T^{1}_1\leq (u-U)\wedge (v-V) , T^{2}_{1}\leq v-V,  U\leq u \right),\intertext{where $(U,V)$ is independent of $(T^{1}_{1},T^{2}_{1})$. So}
			&=\E\left[\1_{\{u\geq  U\}}\psi\left((u-U)\wedge (v-V), v-V\right)\right].
		\end{align*}

To prove part (b), i.e., $\psi\geq \gamma$, define first the function
\[f(u,v):=\limsup_{(x,y)\to (u,v)}\left(\gamma(x,y)-\psi(x,y)\right)~\forall (u,v)\in \Rbar^2~,\]
which is upper-semi continuous and attains its maximum on any closed subset of $\Rbar^2$ (see Appendix~\ref{app:sc}.\textbf{1},\textbf{2}). Let
\[K:=\sup_{(u,v)\in \Rbar^2}f(u,v)~.\]
If $K\leq 0$ the proof is complete; assume otherwise that $K>0$. Define
\[\Kcal:=\{(u,v)\in \Rbar^2: f(u,v)=K\}~, \]
which is a closed subset of $\Rbar^2$, and
\[a:=\min\{u\in \Rbar: \exists v\in \Rbar: (u,v)\in \Kcal \}\]
\[b:=\min\{v\in \Rbar: (a,v)\in \Kcal \}~.\]
Since $f(a,b)=K>0$, there exists a sequence $(a_n,b_n)\in  {\supp (\gamma\vee 0)}$ such that $(a_n,b_n)\to (a,b)$ as $n\to \infty$ and
\begin{align*}K&=f(a,b)=\lim_{n\to \infty}(\gamma-\psi)(a_n,b_n)\\&\leq \limsup_{n\to \infty}\E\left[\1_{\{a_n\geq  U\}}(\gamma-\psi)\left((a_n-U)\wedge (b_n-V), b_n-V\right)\right]\intertext{Since $\gamma-\psi\leq K_\gamma<\infty$, we can further use Fatou's lemma}&\leq \E\left[\limsup_{n\to \infty}\1_{\{a_n\geq  U\}}(\gamma-\psi)\left((a_n-U)\wedge (b_n-V), b_n-V\right)\right]\\&\leq K\cdot\P(a=U) +\E\left[\1_{\{a> U\}}f\left((a-U)\wedge (b-V), b-V\right)\right]\\&\leq K\cdot\P(a\geq  U)~,\end{align*}
using the definitions of $K$ and $f$. It then follows that $\P(a\geq U)=1$, such that necessarily
\begin{equation} \label{abinf}
f\left((a-U)\wedge (b-V), b-V\right)=K~
\end{equation}
holds almost surely (a.s.) for the inequalities above to hold as equalities.

Next we claim that $(a,b)=(\infty,\infty)$. Assume by contradiction that $a<\infty$. Then
\eqref{abinf} and  $\P(U>  0)>0$ contradict with the choice of $a$, and hence $a=\infty$. Similarly, assume by contradiction that $b<\infty$. Then \eqref{abinf} and $\P(V>0)>0$ contradict with the choice of $b$, and hence $b=\infty$ as well.
Finally, \[K=	f(\infty, \infty)=\limsup_{u,v\to \infty}(\gamma-\psi)(u,v)=0\]
from the limiting conditions on $\gamma$ and $\psi$, and hence $\psi\geq \gamma$.

We can now prove the uniqueness of $\psi$ solving for (\ref{equationpsi}). Let $\psi_1$ and $\psi_2$ be two bounded solutions satisfying $\psi_i(\infty,\infty)=\lim_{u,v \to \infty}\psi_i(u,v)=1$. Applying the second part of the theorem with $\psi=\psi_i$ and $\gamma=\psi_{3-i}$ (note that the proof only needs that $\psi$ satisfies (\ref{equationpsi}), is bounded, and $\psi(\infty,\infty)=\lim_{u,v \to \infty}\psi(u,v)=1$) we obtain that $\psi_i\geq \psi_{3-i}$ for $i=1,2$, and hence $\psi_1=\psi_2$.

The proof for the last part of the theorem, i.e., $\psi(u,v)\leq \eta(u,v)$ on $v\geq u$ is similar; see Appendix~\ref{app:3rdpart}.
	\end{proof}

We are now able to make the connection between the generic functions $\gamma$ and $\eta$ from Parts (b) and (c) of Theorem~\ref{th:psigammaeta}, and bounds on $\P(W>x)$.

\begin{corollary}{\sc{(Generic Upper and Lower Bounds)}}\label{cor:Wgammaeta}
Consider the functions $\psi$, $\gamma$, and $\eta$ as in Theorem~\ref{th:psigammaeta}. Then the waiting time of a job $n\to\infty$ satisfies for all $x\geq 0$
\begin{align}
1-\E\left[\eta(x-(Z-Y)_+, x-(Z-Y) )\right]&\leq \P(W>x)=1-\E\left[\psi(x-(Z-Y)_+, x-(Z-Y) )\right]\notag\\&\quad\leq 1-\E\left[\gamma(x-(Z-Y)_+, x-(Z-Y))\right]~.
\end{align}
The corresponding sojourn time satisfies
\begin{align}
&1-\E\left[\1_{\{Z_1+Y<x\}}\eta(x-(Z_1+Z_2\vee Y),x-(Z_1+Z_2)) \right]\notag\\&
\qquad\leq \P(D>x)=1-\E\left[\1_{\{Z_1+Y<x\}}\psi(x-(Z_1+Z_2\vee Y),x-(Z_1+Z_2)) \right]\notag\\&\qquad\qquad\leq 1-\E\left[\1_{\{Z_1+Y<x\}}\gamma(x-(Z_1+Z_2\vee Y),x-(Z_1+Z_2)) \right]~.
\end{align}
\end{corollary}

Note that we omitted indexes in the $Y$'s and $Z$'s, where possible, to avoid clutter. This corollary hints that explicit bounds on $\P(W>x)$ and $\P(D>x)$ can be obtained once constructing explicit functions $\gamma$ and $\eta$ satisfying the conditions from parts (b) and (c) of Theorem~\ref{th:psigammaeta}, respectively. An example will be provided in~\ref{sec:example}.

\begin{proof}
From $W$'s representation from (\ref{eq:W2nodes}) it follows for all $x\geq0$
\[\begin{aligned}\P(W>x)&=\P(\max\left\{0,T^1+(Z-Y)_+, T^2+Z-Y\right\}>x)\\&=1-\P(\max\left\{0,T^1+(Z-Y)_+, T^2+Z-Y\right\}\leq x)\\&=1-\P\left(T^1\leq x- (Z-Y)_+, T^2\leq x- (Z-Y)\right)\\&=1-\E\left[\psi(x-(Z-Y)_+, x-(Z-Y))\right]~.\end{aligned}\]
Since $x-(Z-Y)_+\leq  x-(Z-Y)$ and $\gamma(u,v)\leq\psi(u,v)\leq \eta(u,v)~\forall v\geq u$, the upper and lower bounds on $P(W>x)$ follow immediately. The bounds on $\P(D>x)$ follow from $D=W+Z_1+Y$.
\end{proof}

\subsection{The Integral Equation (\ref{equationpsi}) vs Related Work}\label{sec:ierw}
At the core of Theorem~\ref{th:psigammaeta}, which enables the generic construction of upper and lower bounds on $P(W>x)$, stands the integral equation (\ref{equationpsi}). This can be rewritten as a two-dimensional renewal equation
\begin{equation}\psi(u,v)=\int_{-\infty}^u\int_{-\infty}^{\infty}\psi\left((u-x)\wedge (v-y),v-y\right)dF_{U,V}(x,y)~,\label{eq:renewalpsi}\end{equation}
where $F_{U,V}$ is the joint distribution of $(U,V)$; the `renewalness' stems from expressing the underlying random walks in $T_1^1$ and $T_1^2$ in terms of $T_2^1$ and $T_2^2$, by extracting the first increments $U_1$ and $V_1$ and further using stationarity.

From a conceptual point of view, \eqref{eq:renewalpsi} relates to the standard analysis for $GI/G/1$. Indeed, solving for the waiting time $\P(W^1\leq x)$, for the first queue from Fig.~\ref{fig:2queues}, reduces to solving for
\begin{equation}\P(W^1\leq x)=\int_{-\infty}^x\P(W^1\leq x-y)d F_U(y)~,\label{eq:gg1renewal}\end{equation}
by applying the renewal argument and  Lebesgue Convergence Theorem.

The crucial difference between \eqref{eq:gg1renewal} and \eqref{eq:renewalpsi} stands in the integrand itself. While \eqref{eq:gg1renewal} uses the very distribution $\P(W^1\leq x)$ which is being sought after, the integrand in \eqref{eq:renewalpsi} is based on the joint distribution $\P(T^{1}_{1}\leq u, T^{2}_{1}\leq v )$ stemming from the underlying random walks in $W$; as shown in Theorem~\ref{th:psigammaeta} (Part (a)), this joint distribution is also the \textit{unique} solution of \eqref{eq:renewalpsi}.

Despite a vast amount of related literature (see, e.g., Cohen~\cite{cohen1982}) there is no exact and closed-form solution to \eqref{eq:gg1renewal}, partly due to outstanding numerical challenges associated to Wiener-Hopf type of integral equations. The equation does however lend itself to a generic and especially simple construction of upper and lower bounds. Indeed, by assuming the existence of a function $\gamma(x)$ satisfying for all $x\geq0$
\begin{equation}\int_{-\infty}^x \gamma(y)dF(y)\geq\gamma(x)~,\label{eq:gammmagg1k}\end{equation} then
$$\P(W^1>x)\leq 1-\gamma(x)~.$$ The proof is immediate using the same renewal argument and induction (see Kingman~\cite{Kingman70}). One such function is $\gamma(x)=1-e^{-\theta x}$, where $\theta>0$ satisfies $\E[e^{\theta U}]=1$ (or, more generally, $\theta=\sup\{r>0:\E[e^{r U}]\leq 1\}$; recall the discussion about light-tailed from the start of \ref{sec:m2}); the corresponding bound $\P(W^1>x)\leq e^{-\theta x}$ is known as the Kingman's bound for $GI/G/1$ queues which can alternatively be obtained using a martingale-based argument~(Kingman~\cite{Kingman64}).

It is instructive to apply our integral equation for a single $GI/G/1$ queue, by letting $Z=0$ (i.e., instantaneous service times at the second queue in the tandem from Fig.~\ref{fig:2queues}). Then, according to part (b) from Theorem~\ref{th:psigammaeta} and Corollary~\ref{cor:Wgammaeta}, the derivation of an upper bound reduces to finding a function $\gamma(u,v)$ satisfying
\begin{equation}\E\left[\1_{\{u\geq Y-X\}}\gamma((u-(Y-X))\wedge(v+X),v+X)\right]\geq\gamma(u,v)~,\label{eq:gammmagg1s}\end{equation}
for all $(u,v)\in \supp(\gamma\vee 0)\subseteq \Rbar^2$ and $\gamma(\infty,\infty)=1$. As expected, Kingman's bound is recovered by letting
$$\gamma(u,v)=\left\{\begin{array}{ll}1-e^{-\theta u}&\textrm{if}~u\leq v\\1&\textrm{otherwise}\end{array}\right.~.$$
Similarly, both \eqref{eq:gammmagg1k} and \eqref{eq:gammmagg1s} recover the tighter Ross' bound $\P(W^1>x)\leq\frac{1}{\inf_{u\geq0}\E\left[e^{\theta (U-u)}\mid U> u\right]}e^{-\theta x}$, which is exact in the $GI/M/1$ case (Ross~\cite{Ross74}); this can be done by multiplying the exponential in $\gamma(x)$ and $\gamma(u,v)$ by the prefactor from the bound.

For the analysis of a single queue, \eqref{eq:gammmagg1s} is seemingly unnecessarily complex, partly because it stems from the integral equation (\ref{equationpsi}) designed for a tandem of two queues, as opposed to \eqref{eq:gammmagg1k} which is designed for a single queue. Its crucial benefit, and especially of its parent integral equation (\ref{equationpsi}) as well, is that they allow for closed-form solutions -- to be shown next for two queues and later for general tandems.

\subsection{Existence of Poly-Exp Bounds}

Before explicitly constructing $\gamma$ which can lend themselves to sharp (upper) bounds on $\P(W>x)$, we first prove its poly-exp existence along with an explicit construction. For this very purpose, in the proof of the next existence theorem, we are not concerned \textit{yet} with the sharpness of the polynomial's coefficients.

\begin{theorem}{\sc{(Existence of Poly-Exp Upper Bounds)}}\label{th:existence2}
	Define
	\[\theta_1:=\sup\{r>0:\E[e^{r U}]\leq 1\},\qquad \theta_2:=\sup\{r>0:\max\{\E[e^{r V}], \E[e^{r U}]\}\leq 1\} \]\[I_U(r):=\left\lbrace\begin{split} 1&\quad\textrm{if~}~~\E[e^{r U}]= 1\\0&\quad \textrm{otherwise}\end{split}\right.,\qquad I_V(r):=\left\lbrace\begin{split} 1&\quad\textrm{if~}~~\E[e^{r V}]= 1\\0&\quad \text{otherwise}\end{split}\right.\]
	\[\]
	for random variables $U,V$. For some $a\in \R$ define
	\[\mathcal{D}_a:=\{(u,v):u\geq  -a_+,~v\geq (-a)\vee u\}\] and assume that there exists  a constant $K$ such that for all $v\geq 0$
	\[\E\left[(V-v)e^{\theta_2 (V-v)}\mid V>v\right]\leq K<\infty~.\]

	Then there exist a positive constant   $P_1\geq 0$   and a polynomial $P_2:\R^{2}\to \R$ of degree  $I_V(\theta_2)$ and satisfying $P_2(u,v)\geq0~\forall (u,v)\in\Dcal_a$, such that
	\[ \gamma(u,v):=\1_{\{ (u,v)\in \Dcal_a\}}\left[ 1-P_1 e^{-\theta_1 u}-P_2(u,v)e^{-\theta_2 v}\right]\]
	satisfies the requirements from part (b) of Theorem~\ref{th:psigammaeta}.
\end{theorem}

The role of the parameter $a$ is to optimize $\gamma$ in the sense of tightening the bounds for $\P(W>x)$; for more details see the example from~\ref{sec:example}. Note that the poly-exp structure of $\gamma$ involves a mix of two exponentials and corresponding polynomials. If the queues are homogeneous (same law for $U$ and $V$), then $\theta_1=\theta_2$ and hence a single exponential. Otherwise, the poly-exp structure is more nuanced and depends on the location of the `bottleneck'. For instance, if the distributions of $U$ and $V$ are in the same class, but $E[U]>E[V]$, then the first queue is the bottleneck, $\theta_1=\theta_2$, and $I_V(\theta_2)=0$, i.e., the degree of $P_2$ is $0$. Otherwise, if $E[U]<E[V]$, then the second node is the bottleneck; under the additional condition $\E[e^{\theta^+ U}]\geq 1$ (recall the description on `light-tailedness' from the beginning of~\ref{sec:m2}), then $\theta_1>\theta_2$ and $I_V(\theta_2)=1$; thus, the poly-exp structure involves two exponentials whereas the degree of $P_2$ is $1$.

We also mention that the existence of a matching poly-exp structure for $\eta$, needed for lower bounds on $\P(W>x)$, is still open.

\begin{proof} Fix $a\in\R$. We proceed in two steps.

	\textit{Step 1: } First we prove that there exist a constant $Q_1\geq 0$ and a polynomial $Q_2:\R\to \R$ of degree $I_V(\theta_2)$ with non-negative coefficients such that for all $u\geq -a_+, v\geq 0$
	\begin{equation} \label{Q1Step1}
		Q_1e^{-\theta_1 u}\geq \E\left[\1_{\{u\geq  U\}}Q_1e^{\theta_1 (U-u)}\right]+\P(U>u)
	\end{equation}
\begin{equation} \label{Q2Step1}
		Q_2(v)e^{\theta_2 (a-v)}\geq \E\left[\1_{\{v\geq  V\}}\left(Q_1e^{\theta_1 (V+a-v)}+Q_2(v-V)e^{\theta_2 (V+a-v)}\right)\right]+\P(V>v)~.
			\end{equation}

	\textit{Proof: }
	Inequality \eqref{Q1Step1} holds immediately for
	\[Q_1:=\left(\inf_{u\geq -a_+}\E\left[e^{\theta_1 (U-u)}\mid U> u\right]\right)^{-1}\]
by splitting $\1_{\{u\geq  U\}}=1-\1_{\{U>u\}}$ and rewriting $\E[\1_{\{U>u\}}e^{\theta_1(U-u)}]=E[e^{\theta_1(U-u)}\mid U>u]\P(U>u)$.
	
	Let $Q_2(v):=A_0+A_1v$. To also prove \eqref{Q2Step1} it is sufficient to show  that there exist the non-negative constants $A_0, A_1$ such that
	\begin{equation}\label{Al-psi-A0}
		\begin{split}
			&	A_0e^{\theta_2 a}\left\lbrace e^{-\theta_2 v}-\E\left[\1_{\{v\geq  V\}}e^{\theta_2 (V-v)}\right]\right\rbrace+A_1e^{\theta_2 a}\left\lbrace v e^{-\theta_2 v}-\E\left[\1_{\{v\geq  V\}}(v-V)e^{\theta_2 (V-v)}\right]\right\rbrace
			\\&\geq \E\left[\1_{\{v\geq  V\}}Q_1e^{\theta_1(V+a-v)}\right]+\P(V>v)~.
		\end{split}
	\end{equation}
	For $v\geq 0$, and using that $\theta_1\geq \theta_2$, the right side can be bounded as
	\begin{align*}
		&\E\left[\1_{\{v\geq V\}}Q_1e^{\theta_1(V+a-v)}\right]+\P(V>v)\\&\leq e^{\theta_1 a}\E\left[\1_{\{v\geq V\}}Q_1e^{\theta_2(V-v)}\right]+\P(V>v)\\	&=  e^{\theta_1 a}\E\left[Q_1e^{\theta_2(V-v)}\right]+\P(V> v) - e^{\theta_1 a}\E\left[Q_1e^{\theta_2(V-v)}\mid V> v\right]\P(V> v)\\&\leq Q_1\E[e^{\theta_2 V}]e^{\theta_1 a-\theta_2 v}+\P(V> v)~.
	\end{align*}
	On the left side, we bound the coefficient of $A_0e^{\theta_2 a}$ in the opposite direction
	\begin{align*}
		e^{-\theta_2 v}-\E\left[\1_{\{v\geq  V\}}e^{\theta_2 (V- v)} \right] \geq \left(1-\E\left[e^{\theta_2 V}\right]\right)e^{-\theta_2 v}+ \P(V> v)
	\end{align*}
	and similarly the coefficient of $A_1$
	\begin{align*}
	&ve^{-\theta_2 v}-\E\left[\1_{\{v\geq  V\}}\left(v-V\right)e^{\theta_2 (V- v)} \right]\\&\qquad=ve^{-\theta_2 v}-\E\left[\left(v-V\right)e^{\theta_2 (V- v)}\right] +\E\left[\1_{\{V> v\}}\left(v-V\right)e^{\theta_2 (V- v)}\right]\\&\qquad\geq \left(1-\E\left[e^{\theta_2 V}\right]\right)ve^{-\theta v}+\E\left[Ve^{\theta_2 V}\right]e^{-\theta_2 v}-K\P(V> v)~.
	\end{align*}

Therefore, it is sufficient to determine the coefficients $A_0$ and $A_1$ satisfying the tighter version of~\eqref{Q2Step1}

	\begin{equation}\label{Al-psi-A0tight}
		\begin{split}
			&	A_0e^{\theta_2 a}\left\lbrace \left(1-\E\left[e^{\theta_2 V}\right]\right)e^{-\theta_2 v}+ \P(V> v)\right\rbrace\\&\qquad\qquad\qquad+A_1e^{\theta_2 a}\left\lbrace \left(1-\E\left[e^{\theta_2 V}\right]\right)ve^{-\theta v}+\E\left[Ve^{\theta_2 V}\right]e^{-\theta_2 v}-K\P(V> v)\right\rbrace
			\\&\qquad\qquad\geq Q_1\E[e^{\theta_2 V}]e^{\theta_1 a-\theta_2 v}+\P(V> v)~.
		\end{split}
	\end{equation}

There are two cases. If $\E\left[e^{\theta_2 V}\right]<1$ then
	\[A_1:=0\quad\textrm{and}\quad A_0:=\max\left\{\frac{Q_1\E[e^{\theta_2 V}]e^{(\theta_1-\theta_2)a}}{1-\E[e^{\theta_2 V}]}, e^{-\theta_2 a}\right\}\]
	satisfy \eqref{Al-psi-A0} and $Q_2(v)=A_0$ has degree $I_V(\theta_2)=0$.

In the other case, if $\E\left[e^{\theta_2 V}\right]=1$, then
	\[A_1:=\frac{Q_1e^{(\theta_1-\theta_2)a}}{\E[Ve^{\theta V}]}\quad\textrm{and}\quad A_0:=\max\left\{e^{-\theta_2 a}+A_1K,0\right\}\]
	also satisfy \eqref{Al-psi-A0}; moreover $Q_2(v)=A_0+A_1v$ has degree $I_V(\theta)=1$. We note that, in the expression of $A_0$, only the first term in the `$\max$' is needed for \eqref{Al-psi-A0}; the second is simply needed for $A_0\geq0$.

We make the important remark that, when $\E\left[e^{\theta_2 V}\right]=1$, $Q_2(v)$ has necessarily degree $1$; concretely, $A_1>0$ is needed to neutralize the term $e^{-\theta_2 v}$ from the right side of \eqref{Al-psi-A0}.

	\textit{Step 2: }Let $P_1:=Q_1$ and $P_2(u,v):=Q_2(v+a)$ from Step 1. Then $\gamma$ satisfies \eqref{pre-estimatepsi}.
	
	\textit{Proof: }Note first that the `marginal' function of $\gamma$ for the first queue
	\[\gamma_{1}(u):=(1-P_1e^{-\theta_1 u})\1_{\{u\geq -a_+\}}\]
	satisfies for all $u\geq -a_+$
	\begin{equation}\label{gamma2}
	\begin{split}	\E\left[\1_{\{u\geq U\}}\gamma_1(u-U)\right]&=\E\left[\1_{\{u\geq  U\}}\left(1-P_1e^{\theta_1 (U-u)}\right)\right]\\&\geq 1- P_1e^{-\theta_1 u}=\gamma_1(u)
		\end{split}
	\end{equation}
using the definition of $Q_1$.
	Note also that
	\[\supp(\gamma\vee 0)=\{(u,v)\in \Dcal_a:  1> Q_1e^{-\theta_1 u}+ Q_2(v+a)e^{-\theta_2 v}\} \subseteq \Dcal_a~.\]
	We can now prove condition \eqref{pre-estimatepsi} for $\gamma$. For all  $(u,v)\in \Dcal_a$
	\begin{align*}
		&\E\left[\1_{\{u\geq   U\}}\gamma\left((u-U)\wedge (v-V),  v-V \right)\right]\\&=\E\bigg[\1_{\{u\geq  U,v+a\geq V\}} \big[1-Q_1e^{-\theta_1\left( (v-V)\wedge(u-U)\right)} -Q_2(a+v-V)e^{-\theta_2 (v-V)}\big]\bigg]\intertext{Since $ Q_1\geq 0$ and using $\max\{x,y\}\leq x+y$ for $x,y\geq0$ we can continue with}&\geq \E\Bigg[\1_{\{u\geq  U,v+a\geq  V\}} \bigg(1-Q_1 e^{\theta_1 (U-u)} -\big(Q_1e^{\theta_1 (V-v)}+Q_2(a+v-V)e^{\theta_2 (V-v)}\big)\bigg)\Bigg]\\&=	\E\bigg[\1_{\{u\geq  U, v+a< V\}} \big(-1+Q_1  e^{\theta_1 (U-u)}\big)\bigg] +	\E\bigg[\1_{\{u\geq  U\}}\big(1-Q_1  e^{\theta_1 (U-u)}\big)\bigg]\\&\quad -\E\left[\1_{\{u\geq  U, v+a\geq  V\}}\big(Q_1e^{\theta_1 (V-v)}+Q_2(a+v-V)e^{\theta_2 (V-v)}\big) \right]\intertext{Using $Q_1\geq 0$ for the first expectation and \eqref{gamma2} for the second, since $(u,v)\in \Dcal_a$ implies   $u\geq -a_+$, we can further continue}&\geq  -\P\left(u\geq  U, v+a< V \right)+\left(1-Q_1e^{-\theta_1 u}\right)\\&\quad -\E\left[\1_{\{u\geq  U, v+a\geq  V\}}\big(Q_1 e^{\theta_1 (V-v)}+Q_2(a+v-V) e^{\theta_2 (V-v)}\big)\right]\\&\geq  -\P\left(v+a<V \right) -\E\left[\1_{\{ v+a\geq  V\}}\big(Q_1 e^{\theta_1 (V-v)}+Q_2(a+v-V) e^{\theta_2 (V-v)}\big)\right]\\&\quad+\left(1-Q_1e^{-\theta_1 u}\right)
		\\&\geq 1-Q_1e^{-\theta_1 u}-Q_2(v+a)e^{-\theta_2 v}=\gamma(u,v)~.
	\end{align*}
	In the second inequality we discarded the event $\{u\geq U\}$ and in the last we used \eqref{Q2Step1} from Step 1; for the last equality recall that the entire chain of inequalities holds for $(u,v)\in \Dcal_a$.

Finally, the remaining conditions from part (b) of Theorem~\ref{th:psigammaeta}, i.e., $\gamma$ is bounded and $\gamma(\infty,\infty)=1$ hold trivially.
\end{proof}

\section{General Case: $M$ Queues}\label{sec:gc}
We now address the general case of $M$ queues in tandem $GI/G/1\rightarrow \cdot/G/1\rightarrow\dots\rightarrow\cdot/G/1$. As in the $M=2$ case, job $0$ arrives at time zero and job $n$ arrives at time $X_0+\cdots+X_{n-1}$. The service times of job $n$ at the nodes are light-tailed and denoted by $Y^{(j)}_{n+j-1}$ for $j=\{1,2,\ldots,M\}$\footnote{Superscript indexes mainly stand for the queues' positions in the tandem.}; see Fig.~\ref{fig:shapem}. All sequences are i.i.d. and mutually independent, and we assume the stability condition $\E[X_0]>\max_j\E[Y^{(j)}_{j-1}]$.

	\begin{figure}[h]
		\centering
		\includegraphics[width=0.7\linewidth]{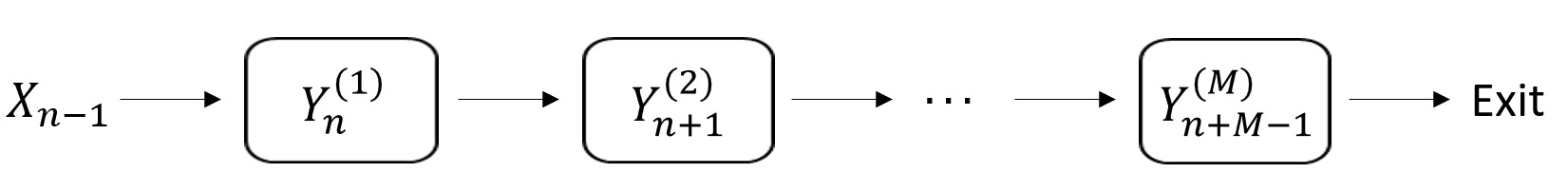}
		\caption{Inter-arrival and service times for job $n$ in a tandem of $M$ queues}
		\label{fig:shapem}
	\end{figure}

The exit time of job $n$ from the tandem is
		\[\begin{split}\max_{0\leq i_1< \ldots< i_{M}\leq M+n-1}&X_0+\cdots+X_{i_1-1}+Y^{(1)}_{i_1}+\cdots+Y^{(1)}_{i_2-1}+Y^{(2)}_{i_2}+\cdots+Y^{(2)}_{i_3-1}\\&+\cdots+Y^{(M)}_{i_{M}}+\cdots+Y^{(M)}_{M+n-1}\end{split}~.\]
		Changing all sub-indexes from $k$ to $M+n-k$ the exit time can be rewritten as
		\[\begin{split}\tau=\max_{1\leq i_M< i_{M-1}< \ldots< i_{1}\leq M+n}& Y^{(M)}_{1}+\cdots+Y^{(M)}_{i_{M}}+Y^{(M-1)}_{i_{M}+1}+\cdots+Y^{(M-1)}_{i_{M-1}} \\&+\dots+Y^{(1)}_{i_2+1}+\dots+Y^{(1)}_{i_1}+X_{i_1+1}+\dots+X_{M+n}\end{split}~.\]

Therefore, the end-to-end waiting time of job n has the same distribution as
	\[\begin{split}
		W&=\tau-[Y^{(M)}_1+Y^{(M-1)}_2+\cdots+Y^{(1)}_{M}+X_{M+1}+X_{M+2}+\cdots X_{M+n}]\\&=\max_{1\leq j\leq m}\left\{T^{j}_{M+1,M+n}+S^{j}-\sum_{k=1}^M Y^{(M+1-k)}_k\right\}\vee 0
	\end{split}\]
		where

\[S^{M}=Y^{(M)}_{1}+\cdots+Y^{(M)}_{M}~,\]
		\[S^{j}=\max_{1\leq  i_{M}< \ldots< i_{j+1}\leq M}Y^{(M)}_{1}+\cdots+Y^{(M)}_{i_{M}}+Y^{(M-1)}_{i_{M}+1}+\cdots+Y^{(M-1)}_{i_{M-1}}+\cdots+Y^{(j)}_{i_{j+1}+1}+\cdots+Y^{(j)}_{M}~,\] for $1\leq j\leq M-1$, and
		\[ T^{j}_{k,l}:=\sup_{k\leq  i_j<i_{j-1}<\cdots < i_1\leq l}V^{(j)}_k+\cdots+V^{(j)}_{i_j}+V^{(j-1)}_{i_{j}+1}+\cdots+V^{(j-1)}_{i_{j-1}}+\cdots+V^{(1)}_{i_{2}+1}+\cdots +V^{(1)}_{i_1}\]
		with $V^{(i)}\simeq  Y^{(i)}-X$.

Relative to the case $M=2$ from~\ref{sec:m2}, the expression of $W$ is structurally identical but written in a general form in terms of the $S^j$'s. Instantiated for $M=2$, $S^{j}-\sum_{k=1}^M Y^{(M+1-k)}_k$ recovers $(Y_2^{(2)}-Y_2^{(1)})_+$ when $j=1$ and $Y_2^{(2)}-Y_2^{(1)}$ when $j=2$; note also from Figs.~(\ref{fig:shapem}) and (\ref{fig:2queues}) that $(Y^{(1)}_n)$ and $(Y^{(2)}_n)$ herein correspond to $(Y_n)$ and $(Z_n)$, respectively, from~\ref{sec:m2}.

The next three results are generalizations of the corresponding results from the $M=2$ case; as their proofs are similar to those from~\ref{sec:m2}, we defer them to Appendix~\ref{app:proofsgc}.

	\begin{theorem}{\sc{(Generic Construction of Upper and Lower Bounds)}}\label{th:psigammaetaM}
		Let $V^{(i)}$  be random variables satisfying $\P(V^{(i)}>0)>0$ for $i=1,\dots,M$, and let $(V^{(1)}_1,\ldots, V^{(M)}_1), (V^{(1)}_2,\ldots, V^{(M)}_2), \ldots$ be i.i.d. copies of $(V^{(1)},\ldots, V^{(M)})$. Denote
		\[\Vcal_M(v_1,\ldots, v_M):=\left(\bigwedge_{1\leq i\leq 2}\left(v_i-V^{(i)}\right),  \ldots,\bigwedge_{M-1\leq i\leq M}\left(v_i-V^{(i)}\right), v_M-V^{(M)} \right)~.\]
		\begin{enumerate}[label=(\alph*)]
\item
The integral equation
		\begin{equation}\label{psimnode}
		\begin{split}
		\E\left[\1_{\{v_1\geq  V^{(1)}\}}\psi_M\left(\Vcal_M(v_1,\ldots, v_M) \right)\right]= \psi_M(v_1,\ldots,v_M)
	\end{split}
		\end{equation}
admits a unique solution in the class of bounded functions on $\Rbar^M$ having the limit $\psi_M(\infty,\dots,\infty)$ \\$=\lim_{v_1,\dots,v_M\to \infty}\psi_M(v_1,\dots,v_M)=1$. This is given by
$$\psi_M(v_1,\ldots, v_M):=\P\left(T^1_1\leq v_1,\ldots, T^M_1\leq v_M\right)~,$$  where
	 \[ T^{j}_{k}:=\sup_{k\leq  i_j<i_{j-1}<\cdots < i_1<\infty}V^{(j)}_k+\cdots+V^{(j)}_{i_j}+V^{(j-1)}_{i_{j}+1}+\cdots+V^{(j-1)}_{i_{j-1}}+\cdots+V^{(1)}_{i_{2}+1}+\cdots +V^{(1)}_{i_1}~,\]
for $j=1,\dots,M$ and $k\in\N$.

	\item Assume that the function $\gamma_M:\Rbar^M\to (-\infty,K_{\gamma_M}]$, for some finite $K_{\gamma_M}$, satisfies for all $(v_1,\ldots, v_M)\in \supp(\gamma_M\vee 0)\subseteq \Rbar^M$
	\begin{equation}\label{gammamnode}
	\begin{split}
		\E\left[\1_{\{v_1\geq  V^{(1)}\}}\gamma_M\left(\Vcal_M(v_1,\ldots, v_M)\right)\right]\geq  \gamma_M(v_1,\ldots,v_M)~.
	\end{split}
\end{equation}
If $\gamma_M(\infty,\ldots,\infty)=\limsup_{v_1,\dots,v_M\to \infty}\gamma_M(v_1,\dots,v_M)=1$ then $\psi_M\geq \gamma_M$.
\item  Assume that the function $\eta_M:\Rbar^M\to[0,\infty)$ satisfies for all  $(v_1,\ldots, v_M)\in \supp(\psi_M)$
\begin{equation}\label{etaamnode}
	\begin{split}
		\E\left[\1_{\{v_1\geq  V^{(1)}\}}\eta_M\left(\Vcal_M(v_1,\ldots, v_M) \right)\right]\leq   \eta_M(v_1,\ldots,v_M)
	\end{split}
\end{equation}
If $\eta_M(\infty,\ldots,\infty)=\liminf_{v_1,\dots,v_M\to \infty}\eta_M(v_1,\dots,v_M)=1$ then $\psi_M\leq \eta_M$.
		\end{enumerate}
			\end{theorem}

The next corollary links the functions $\psi_M$, $\gamma_M$, and $\eta_M$ to $\P(W>x)$.

\begin{corollary}{\sc{(Generic Upper and Lower Bounds)}}\label{cor:WgammaetaM}
Consider the functions $\psi_M$, $\gamma_M$, and $\eta_M$ as in Theorem~\ref{th:psigammaetaM}. Then the waiting time of a job $n\to\infty$ satisfies for all $x\geq 0$

\[\begin{split}&1-\E\left[\eta_M\left(x+\sum_{k=1}^M Y^{M+1-k}_k-S_1, \ldots,x+\sum_{k=1}^M Y^{M+1-k}_k-S_M \right)\right]\leq \P(W> x)\\&\qquad\leq 1-\E\left[\gamma_M\left(x+\sum_{k=1}^M Y^{M+1-k}_k-S_1, \ldots,x+\sum_{k=1}^M Y^{M+1-k}_k-S_M \right)\right].\end{split}\]
\end{corollary}

The last result proves the existence of poly-exp upper bounds.

			\begin{theorem}{\sc{(Existence of Poly-Exp Upper Bounds)}}\label{th:existenceM}
				Fix $m\in \N$. Assume that $\theta_i:=\sup\{r>0:\forall 1\leq j\leq i:\E[e^{r V^{(j)}}]\leq 1\}$  for random variables $V^{(i)}, i\in\{1,2,\ldots,m\}$.
				Let \begin{equation}I_i(\theta):=\left\lbrace\begin{split}1& \qquad \text{if }~\E[e^{\theta V^{(i)}}]=1\\0&\qquad \text{otherwise} \end{split}\right.\label{eq:Iitheta}\end{equation}
				and
					\[d_{1}(\theta):=0, \quad 	d_{j}(\theta):=I_2(\theta)+\ldots+I_j(\theta), \qquad  2\leq j\leq M~. \]
				Let $a\in \R$ and
				\[\Dcal_a^M:=\left\lbrace (v_1, \ldots, v_M):  v_1\geq -a_+, v_2\geq -a, \ldots, v_M\geq -a\right\rbrace~.\]
				Let $\zeta_1:=\theta_1$ and for $2\leq i\leq M$, let $\zeta_i:=\theta_i$ if either $ \theta_{i-1}=\theta_i$ or $d_{i-1}(\theta_{i-1})=0$, otherwise let $\zeta_i\in (\theta_{i}, \theta_{i-1})$. Suppose that for all $ (v_1, \ldots, v_M)\in \Dcal^M_a$ and all $i\in \{1,2,\ldots, M\}$ and $j\in\{1,\ldots,d_{i-1}(\theta_{i-1})\vee d_i(\theta_i)\}\cap (2\Z+1)$
				\[ \E\left[\left(V^{(i)}-v_i\right)^je^{\zeta_i (V^{(i)}-v_i)}\mid V^{(i)}>v_i\right]\leq K_i^j<+\infty~.\]
				Then there exists the polynomials  $ P_i:\R^{i}\to \R$ with degrees $d_i(\theta_i)$, respectively, for $i\in\{1,2,\ldots,M\}$ with
				\[P_{i}(v_1, \ldots, v_i)\geq 0\]
				for all $1\leq i\leq M$ and $ (v_1, \ldots, v_M)\in \Dcal^M_a$ such that
				\[ \gamma_M(v_1,\ldots,v_M):=\1_{\{ (v_1, \ldots, v_M)\in \Dcal^M_a\}}\left[ 1-\sum_{i=1}^{M}P_i(v_1, \ldots, v_i)e^{-\theta_i v_i}\right]\]
				satisfies \eqref{gammamnode}.
			\end{theorem}

We note that, unlike in Theorem~\ref{th:existence2} for $M=2$, there is an additional stronger requirement on the conditional expectations involving the new parameter $\zeta$; this is related to the `lightness' of service times' tails (see the proof for the details). The definition of $\Dcal_a^M$ also involves stronger requirements on the conditional expectations than in Theorem~\ref{th:existence2}; this is due to additional technicalities encountered for $M\geq 3$.

Importantly, the degree of the polynomial $P_M(v_1,\dots,v_M)$, which dictates the behavior of the tail $\P(W>x)$ according to Corollary~\ref{cor:WgammaetaM}, depends on the indicator functions from~\eqref{eq:Iitheta}. At one extreme, if all queues are homogeneous and $\E[e^{\theta^+ (Y-X)}]\geq 1$ (recall the description on `light-tailedness' from the beginning of~\ref{sec:m2}), then the degree of $P_M$ is $M-1$, as all indicators $I_i$ but the first are $1$. At another extreme, if there is a single bottleneck, and $\E[e^{\theta^+ (Y-X)}]\geq 1$, then the degree of $P_M$ depends on the position of the bottleneck: if it comes first then the degree is $0$,  otherwise it is $1$ regardless the position (e.g., second or last). Another extreme is when $\E[e^{\theta^+ (Y-X)}]< 1$, e.g., as for the distribution given in Appendix~\ref{app:vld}, in which case the degree of $P_M$ is $0$.

\section{Explicit Poly-Exp Bounds + Numerics}\label{sec:example}
We return to the $M=2$ case from \ref{sec:m2} and provide explicit solutions for $GI/M/1\rightarrow \cdot/M/1$, i.e., exponential service times (with equal rates at both queues), and further simplify them for both $D/M/1\rightarrow \cdot/M/1$ and $Gamma/M/1\rightarrow \cdot/M/1$, i.e., deterministic or Gamma inter-arrival times. Then we derive alternative (non-asymptotic) bounds using large deviations and lastly show numerical comparisons.

Recall the notation $V=Z-X$, $U=Y-X$, where $Z, Y$ and $X$ are independent, and assume that $\E[e^{\theta U}]=\E[e^{\theta V}]=1$ for some $\theta>0$. Let for $R\in \{Y,Z\}$ and $r\geq 0$
		\[K_i^R( r):=\E[(R-r)^ie^{\theta (R-r)}\mid R> r],  \quad \bar{F}_R(r):=\P(R> r)~.\]
		
According to Theorem~\ref{th:psigammaeta} (Part (b)), and inspired from Theorem~\ref{th:existence2}, we need to find the constants  $a, A, B, C, D$ such that, by defining $\gamma:\Dcal_a:=\{(u,v):u\geq  -a_+,~v\geq (-a)\vee u\}\to \R$
\[\gamma(u,v):=\1_{\{ (u,v)\in \Dcal_a\}}\left[ 1-(A+Bv+Cu)e^{-\theta_2 v}-De^{-\theta u}\right]~, \]
the following inequalities hold: $D\geq0$, $A+Bv+Cu\geq0~\forall (u,v)\in \Dcal_a$, and \eqref{pre-estimatepsi} from Theorem~\ref{th:psigammaeta}, i.e.,

		\begin{equation}
			\begin{split}
				\E\Bigg[&\1_{\{u\geq  U, v+a\geq  V\}}\Big[ 1- D\exp\left\lbrace -\theta\left[\left(u-U\right)\wedge \left(v-V\right)\right]\right\rbrace\\&-\left(A+B\left(v-V\right)+C\left[\left(u-U\right)\wedge \left(v-V\right)\right]\right)\exp\left\lbrace -\theta \left(v-V\right)\right\rbrace\Big]\Bigg]\\&\geq 1-De^{-\theta u}-(A+Bv+Cu)e^{-\theta v}~,\label{eq:gg1gg1cond}
			\end{split}
		\end{equation}
for all $(u,v)\in \supp(\gamma\vee 0)\in \Dcal_a$. We recall that $a\in\R$ is an optimization parameter which turns out to be crucial for the overall numerical sharpness of the bounds, at the expense however of making the underlying calculations more tedious.

To find the parameters $A, B, C, D$, and also $a$, we next adopt a sub-optimal approach in the sense of constructing a sufficient set of inequalities for~\eqref{eq:gg1gg1cond}, which are easier to handle separately, and match the parameters; see Appendix~\ref{app:fppGM}.

For the $GI/M/1\rightarrow \cdot/M/1$ model, i.e., $Y,Z\simeq Exp(\mu)$ and $\E[X]=\frac{1}{\mu\rho}$, let $\theta>0$ such that $\E[e^{\theta( Y-X)}]=1$, or, equivalently, $\E[e^{-\theta X}]=(\mu-\theta)/\mu $; existence follows from the stability condition $\rho<1$. Observing first that
		\[K_i^Y( x)=\frac{\mu\times  i!}{(\mu-\theta)^{i+1}}~,\]
three of the main parameters follow immediately from the inequalities 2, 4, 5, 6, and 7 (as listed in Lemma~\ref{cor:gg1gg1exp} from Appendix~\ref{app:fppGM}):
		\[D=\frac{\mu-\theta}{\mu}\qquad  C=\frac{\theta (\mu-\theta)D}{\mu},\qquad B=C\times \frac{\left(\mu\E\left[Xe^{-\theta X}\right]-\frac{\mu-\theta}{\mu}\right)_+}{1-\mu\E\left[Xe^{-\theta X}\right]}~.\]

The remaining parameters $A$ and $a$ can be obtained by treating separately the remaining inequalities 1, 3, and 8 (again, from the same Lemma~\ref{cor:gg1gg1exp}). From the first two we define

	\[\begin{split} A_1(a)&:=aB+\sup_{x\geq -a_+}\Bigg\lbrace\Bigg( B\E\left[\left(\frac{1}{\mu-\theta}-X\right)e^{-\theta X}\left(e^{-\mu(X+x)}\1_{\{X+x\geq 0\}}+\1_{\{X+x<0\}}\right)\right]\\& +C\E\left[e^{-\theta X}\left(\frac{1}{\mu}e^{-\mu(X+x)}\1_{\{X+x\geq 0\}}+\left(\frac{1}{\mu}-X-x\right)\1_{\{X+x<0\}}\right)\right]\Bigg)\\& \times \left(\E\left[e^{-\theta X}\left(e^{-\mu(X+x)}\1_{\{X+x\geq 0\}}+\1_{\{X+x<0\}}\right)\right]\right)^{-1}\Bigg\rbrace \end{split}\]
and
	\[A_2(a):=a(B+C)+ \frac{e^{-\theta a}(\mu-\theta)}{\mu}+\frac{B+C}{\mu-\theta}-D~.\]
Observing that
	\[\E[(t-Y)\mid t\geq Y>t+a]=\frac{(t-(t+a)_+)e^{\mu(t-(t+a)_+)}}{e^{\mu(t-(t+a)_+)}-1}-\frac{1}{\mu}\]
and
		\[\E[e^{\theta(Y-t)}\mid t\geq Y>t+a]=\frac{\mu\left(e^{(\mu-\theta)(t-(t+a)_+)}-1\right)}{\left(e^{\mu(t-(t+a)_+)}-1\right)(\mu-\theta)}~, \]
we obtain from the eight inequality
	\[A_3(a):=B\left(a+\frac{1}{\mu-\theta}\right) -\min_{\ess\inf (X)\wedge (-a)\leq r\leq -a}\left\lbrace\frac{De^{-\theta a}\left(1-e^{(\theta-\mu) r}\right)+Cr}{1-e^{-\mu r}}\right\rbrace +\frac{C}{\lambda} +\frac{e^{-\theta a}(\mu-\theta)}{\mu}~.\]

We can now define
	\begin{equation}A:= \min_{a\geq 0}(A_1(a)\vee A_2(a))\wedge \min_{a\leq 0}(A_1(a)\vee A_2(a)\vee A_3(a))~.\label{eq:dm2_A}\end{equation}
		Letting %that and $A_2(b_1)\geq A_1(b_1)$ for
		\[a_1:=\arg\min_{a\geq 0}(A_2(a))= \left(\frac{\ln(\theta (\mu-\theta)/\mu (B+C))}{\theta}\right)_+\]
			 and observing that $A_1(a)$ is non-decreasing on $\R$, whereas $A_2(a)$ is non-decreasing on $[a_1,\infty)$ and non-increasing on $[0,a_1]$, the optimal parameter $a$ for the first minimum from the expression of $A$ from \eqref{eq:dm2_A} is
		\begin{align*}
			a:=\arg\min_{a\geq 0}(A_1(a)\vee A_2(a))=\sup\left(\{0\leq a\leq a_1: A_1(a)\leq A_2(a)\}\cup\{0\}\right)~.
		\end{align*}

Lastly, we  apply Corollary~\ref{cor:Wgammaeta} to obtain the bounds on the tail of $W$, i.e.,

		\begin{align*}
			&\P(W> x)=1-	\P\left(T^{2}_{1}\leq x-Z+Y, T^{1}_{1}\leq x- (Z-Y)_+\right)\\&\leq 1-\E\Bigg\lbrace\1_{\{x-Z+Y\geq -a\}}\Big[1-D\exp(-\theta (x- (Z-Y)_+))\\&\quad -(A+B(x- (Z-Y))+C(x- (Z-Y)_+))\exp(-\theta (x- (Z-Y)))\Big]\Bigg\rbrace~.
\end{align*}
We now distinguish two cases. If $x\geq (-a)_+$ then the bound becomes
\begin{align}
			\P(W>x)&\leq\frac{1}{2}e^{-\mu (x+a)}+\frac{D\left((2\mu-\theta)e^{-\theta x}-\mu e^{-\mu x+(\theta-\mu)a}\right)}{2(\mu-\theta)}\notag\\&\quad+C\left(\frac{x\mu^2}{\mu^2-\theta^2}e^{-\theta x}+\frac{\mu(1+(\mu-\theta)a)}{2(\mu-\theta)^2}e^{-\mu x+(\theta-\mu)a}-\frac{\mu}{2(\mu-\theta)^2}e^{-\theta x}\right)
			\notag\\&\quad +A\left(e^{-\theta x}\frac{\mu^2}{\mu^2-\theta^2}-e^{-\mu x+(\theta-\mu)a}\frac{\mu}{2(\mu-\theta)}\right)\notag\\&\quad +B\left(\frac{\mu^2}{\mu^2-\theta^2}x e^{-\theta x}-\frac{2\mu^2\theta}{(\mu^2-\theta^2)^2}e^{-\theta x}+\frac{\mu(1+(\mu-\theta)a)}{2(\mu-\theta)^2}e^{-\mu x+(\theta-\mu)a}\right)~.\label{eq:dm2_boundcase1}\end{align}
Otherwise, if $0\leq x< -a$, the bound becomes
\begin{align} \P(W>x)&\leq 1-\frac{1}{2}e^{\mu(x+a)}+\frac{1}{2}De^{(\mu-\theta)x+\mu a}\notag\\&\qquad+\frac{\mu}{2(\mu+\theta)}e^{(\mu+\theta)a+\mu x}\left(A+B\left(\frac{1}{\mu+\theta}+a\right)+Cx\right)~.\label{eq:dm2_boundcase2}
		\end{align}
We observe that the two bounds on $\P(W>x)$ have a \textit{mixed} poly-exp structure as they involve two exponentials (in $\theta$ and $\mu$) along with corresponding polynomials.

In the $D/M/1\rightarrow \cdot/M/1$ special case when $X$ is a constant we can simplify
		\[A_1(a)=aB+B\left(\frac{1}{\mu-\theta}-X\right)+C\left(\frac{1}{\mu}+(a_+-X)_+\right)~.\]

In the other $Gamma/M/1\rightarrow \cdot/M/1$ special case, i.e., $X\simeq Gamma(\alpha, \beta)$ has a Gamma distribution with shape $\alpha$ and rate $\beta$ and distribution function $F_X(x)=G(x,\alpha, \beta)$, we have
		$\E[Xe^{-\theta X}]=\alpha \beta^\alpha/(\beta +\theta)^{\alpha+1}$ in the expression of $B$ and
		\[\begin{split} A_1(a)&=aB+\sup_{x\geq -a_+}\Bigg\lbrace \Bigg[ \left(\frac{B}{\mu-\theta}+\frac{C}{\mu}\right)\left(\frac{\beta}{\beta+\theta+\mu}\right)^\alpha e^{-\mu x}(1-G(-x, \alpha,\beta+\theta+\mu ))\\&\quad +\left(\frac{B}{\mu-\theta}+\frac{C}{\mu}-Cx\right)\left(\frac{\beta}{\beta+\theta}\right)^\alpha G(-x, \alpha,\beta +\theta) -(B+C)\left(\frac{\alpha\beta^\alpha}{(\beta+\theta)^{\alpha+1}}\right)G(-x, \alpha+1, \beta+\theta)\\&\quad -B\left(\frac{\alpha\beta^\alpha}{(\beta+\theta+\mu)^{\alpha+1}}\right)e^{-\mu x}\left(1-G(-x, \alpha+1, \beta+\theta+\mu)\right)
			\Bigg]\\&\quad \times\left[\left(\frac{\beta}{\beta+\theta+\mu}\right)^\alpha e^{-\mu x}(1-G(-x, \alpha,\beta+\theta+\mu ))+\left(\frac{\beta}{\beta+\theta}\right)^\alpha G(-x, \alpha,\beta +\theta)\right]^{-1}\Bigg\rbrace~. \end{split}\]

\subsection{A Large Deviations / Network Calculus Approach}\label{sec:nc}
Here we present alternative bounds using the standard large deviations / network calculus approach, which crucially relies on the Union Bound
$$\E\left[\max\{X,Y\}\right]\leq \E[X]+\E[Y]\qquad\textrm{or}\qquad\P(A\cup B)\leq\P(A)+\P(B)~,$$
for positive r.v. $X$ and $Y$, or events $A$ and $B$. One advantage of this approach is that it yields the \textit{exact} asymptotic decay rates (e.g., for $\P(D>x)$, see~\cite{Ganesh98}). Another, as shown in several applications of network calculus, is that it enables the analysis of queueing networks with broad classes of arrivals, service times, or scheduling algorithms, and can further lead to the exact asymptotic scaling of sojourn times~(\cite{HLiu03,Book-Chang,CiBuLi05,Fidler06,Jiang08,BuLiCi11}).

The drawback of this class of results is poor numerical tightness, particularly in non-asymptotic regimes (i.e., for finite values of $x$ in the case of $\P(D>x)$). This issue was brought up in the context of the (large-deviations-based) effective bandwidth literature from the late 80's - 90's. Choudhury~\et~\cite{choudhury96squeezing} revealed large numerical discrepancies, of several orders of magnitude, between effective bandwidth results and simulations. More recently, similar numerical issues have been reported about network calculus results concerning single queues only; in addition, it was shown that relying on Kingman's GI/G/1 bound (recall ~\ref{sec:ierw}), as opposed to the Union Bound, can largely fix the issue of numerical tightness in single queues~\cite{PoCi14}.

The large numerical inaccuracies mainly stem from the obliviousness of the Union Bound to the underlying correlations; this issue is particularly pronounced  in non-Poisson/non-memoryless type of events, as indicated by Talagrand~\cite{tala96b}. Moreover, the underlying numerical errors accumulate over an infinite number of applications of the Union Bound, as waiting/sojourn times involve whole sample-paths.

For example, in the context of the waiting-time $W$ formulation from~\eqref{eq:W2nodes}, the large deviations approach proceeds by first computing the bounds
\begin{equation*}\left\{\begin{array}{ll}\P\left(\max_{3\leq  i\leq n+2}U_3+\cdots+U_{i}>x\right)\leq(\beta+\beta^2+\beta^3+\dots)e^{-\theta x}=\frac{\beta}{1-\beta}e^{-\theta x}\\\P\left(\max_{3\leq  i<j\leq n+2}V_3+\cdots+V_{i}+U_{i+1}+\cdots+U_{j}>x\right)\leq(\beta^2+2\beta^3+3\beta^4+\dots)e^{-\theta x}=\frac{\beta^2}{(1-\beta)^2}e^{-\theta x}\end{array}\right.~,\end{equation*}
by repeatedly using the Chernoff and Union Bounds, where $\beta=\E\left[e^{\theta(Y-X)}\right]$ and $\theta>0$ is chosen such that $\beta<1$ (to guarantee the convergence of the infinite series). The former bound concerns the maximum of a random walk, an event characteristic to single queues, and follows by infinitely applying the Union Bound -- at the expense of disregarding correlations within the partial sums $U_3+\dots+U_i$. In turn, the latter concerns an event involving a double-maximum (over $i$ and $j$), which is characteristic to a tandem of two queues; the bound itself follows from a nested infinite application of the Union Bound, while also disregarding correlations within the underlying partial sums. The numerical inaccuracies associated with the former bound, as reported in the single-queues literature, conceivably exacerbate in the case of the latter bound targeting tandem queues.

In the case when $Y$ and $Z$ are exponentials with rate $\mu$, we obtain by applying the Union Bound one more time, along with double integration, that \begin{align}
\P(W>x)&\leq \P\left(\max_{3\leq  i\leq n+2}U_3+\cdots+U_{i}+(Z-Y)_+>x\right)\notag\\
&\qquad+\P\left(\max_{3\leq  i<j\leq n+2}V_3+\cdots+V_{i}+U_{i+1}+\cdots+U_{j}+Z-Y>x\right)\notag\\
&\leq \inf_{\{0<\theta<\mu:\beta<1\}}\left(\frac{\beta}{1-\beta}\frac{2\mu-\theta}{2(\mu-\theta)}+\left(\frac{\beta}{1-\beta}\right)^2\frac{\mu^2}{(\mu-\theta)(\mu+\theta)}\right)e^{-\theta x}~.\label{eq:dm2_ld}
\end{align}
Due to the underlying transcendental nature of the bound, the optimal value of $\theta$ requires a numerical search.

\subsection{Numerical Comparisons}
Here we numerically illustrate the main bounds from~\eqref{eq:dm2_boundcase1}-\eqref{eq:dm2_boundcase2} against the large-deviations-based ones from \eqref{eq:dm2_ld}. We also include simulation results obtained from $10^5$ runs and sample paths of $10^4$ points (the `$n$' from \eqref{eq:dm2_ld}); the running time for this setting is $10^{13}$, due to the event involving a nested `$\max$' requiring itself $10^8$ time.

\begin{figure}[h]
\vspace{-0cm}
%\begin{center}
\shortstack{\hspace{0cm}
\includegraphics[width=0.323\linewidth,keepaspectratio]{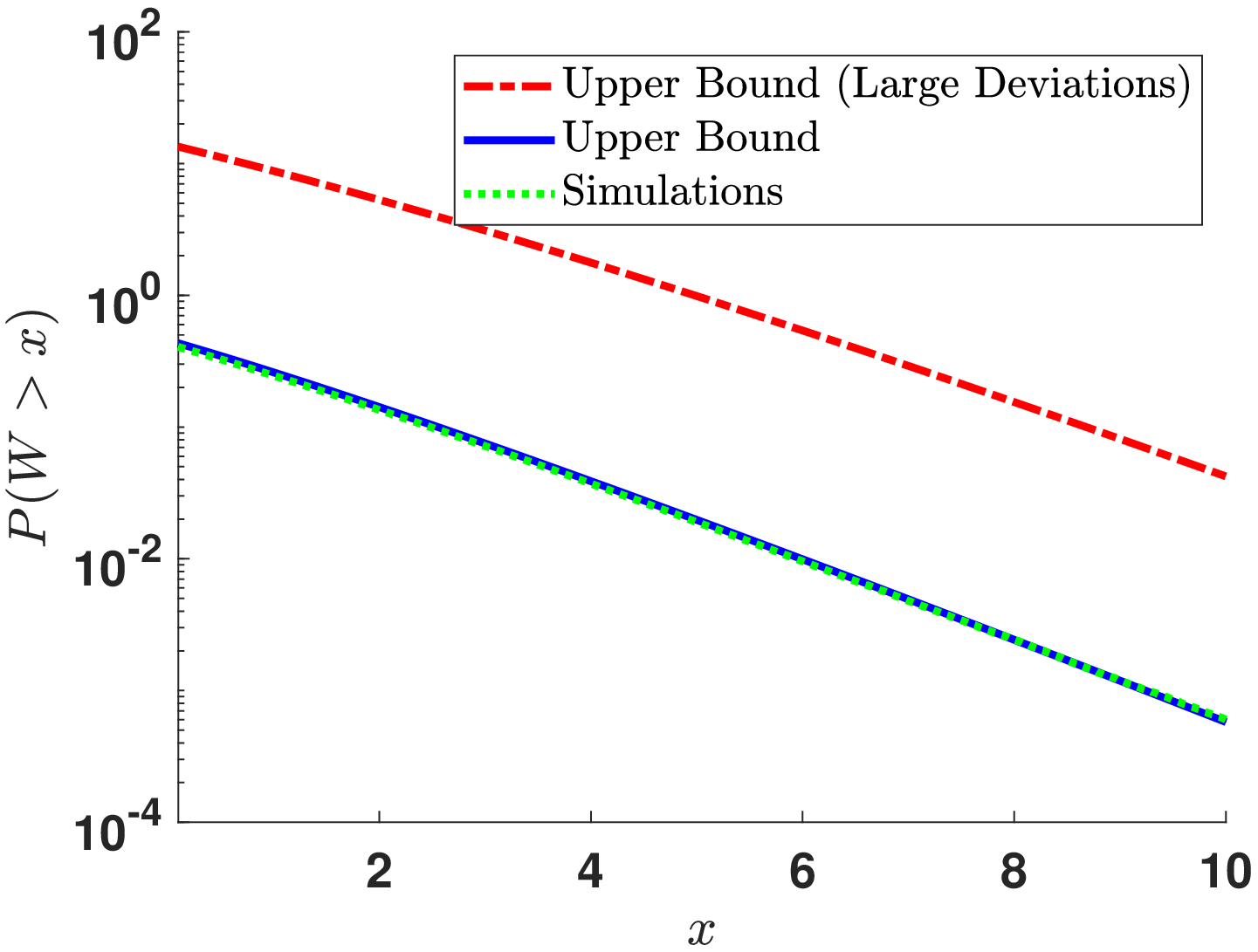}
\\
{\footnotesize (a) $\rho=0.5$} }
\shortstack{\hspace{0cm}
\includegraphics[width=0.323\linewidth,keepaspectratio]{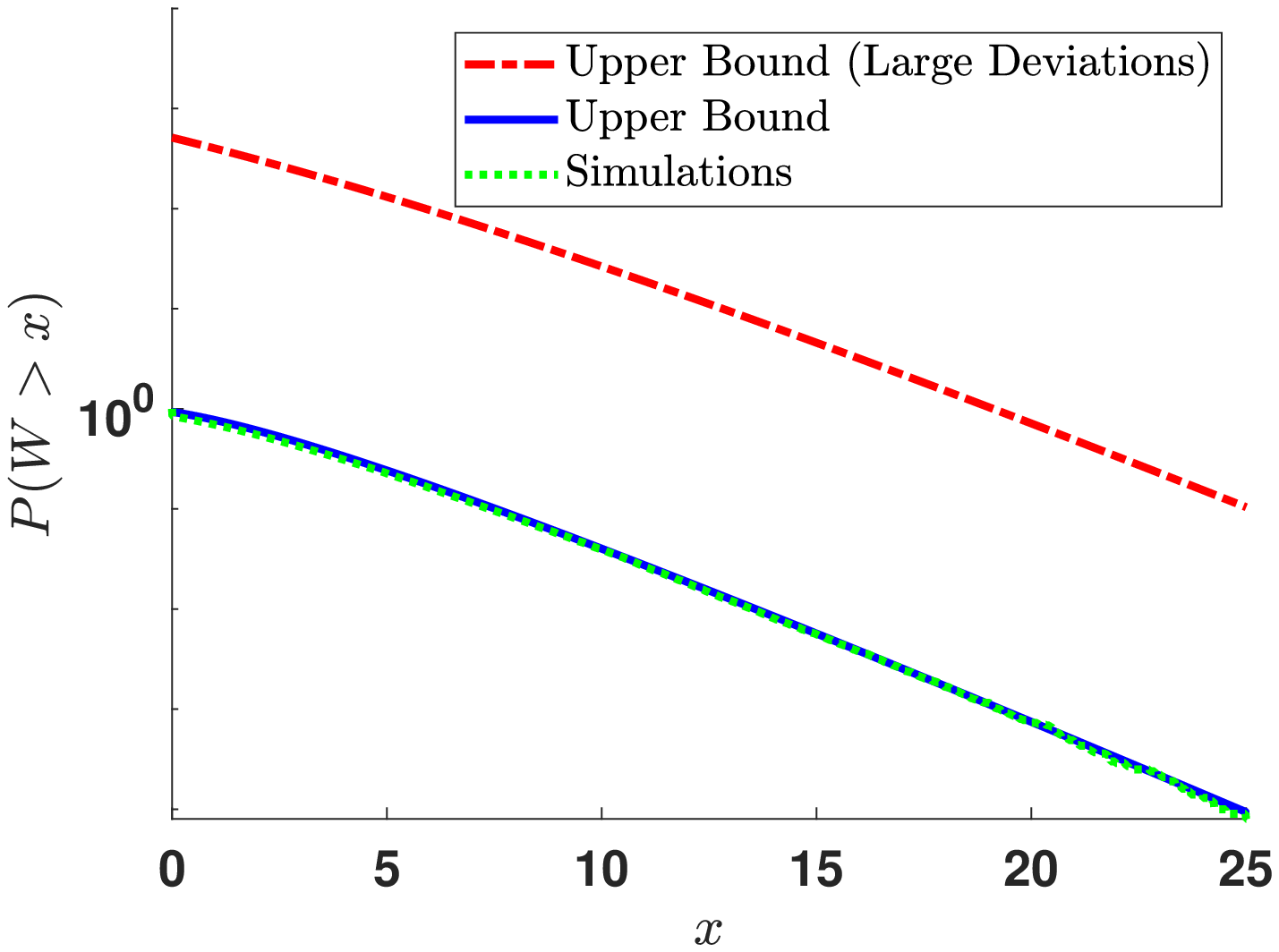}
\\
{\footnotesize (b) $\rho=0.75$}}
\shortstack{\hspace{0cm}
\includegraphics[width=0.323\linewidth,keepaspectratio]{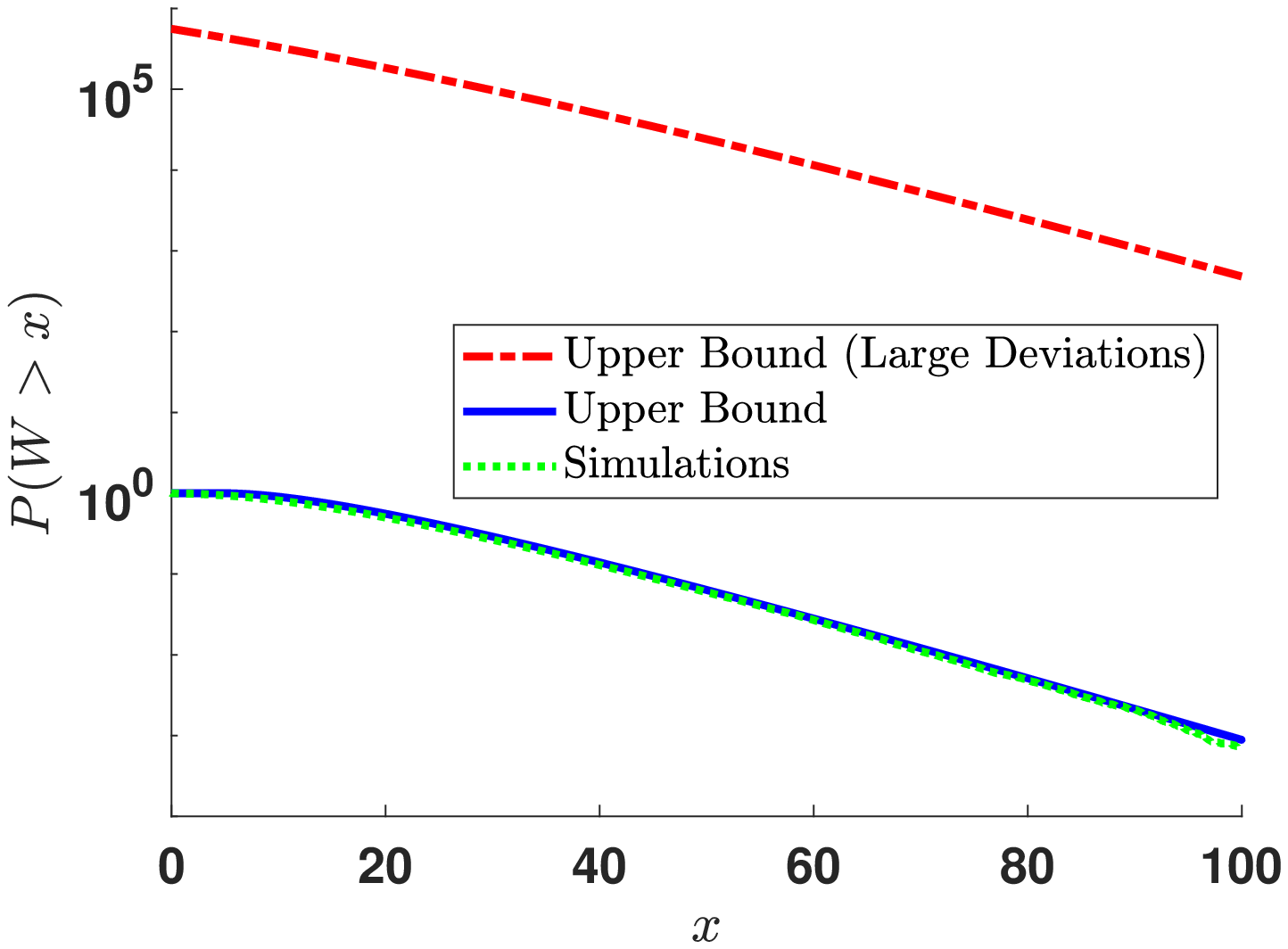}
\\
{\footnotesize (c) $\rho=0.95$} }

 \caption{The waiting time CCDF $\P(W>x)$ for the $D/M/1\rightarrow \cdot/M/1$ tandem: Upper Bounds (from \eqref{eq:dm2_boundcase1}-\eqref{eq:dm2_boundcase2}) vs. Upper Bounds (from \eqref{eq:dm2_ld}, based on large deviations) vs. Simulations} \label{fig:DM2}
%\end{center}
\end{figure}

\begin{figure}[h]
\vspace{-0cm}
%\begin{center}
\shortstack{\hspace{0cm}
\includegraphics[width=0.323\linewidth,keepaspectratio]{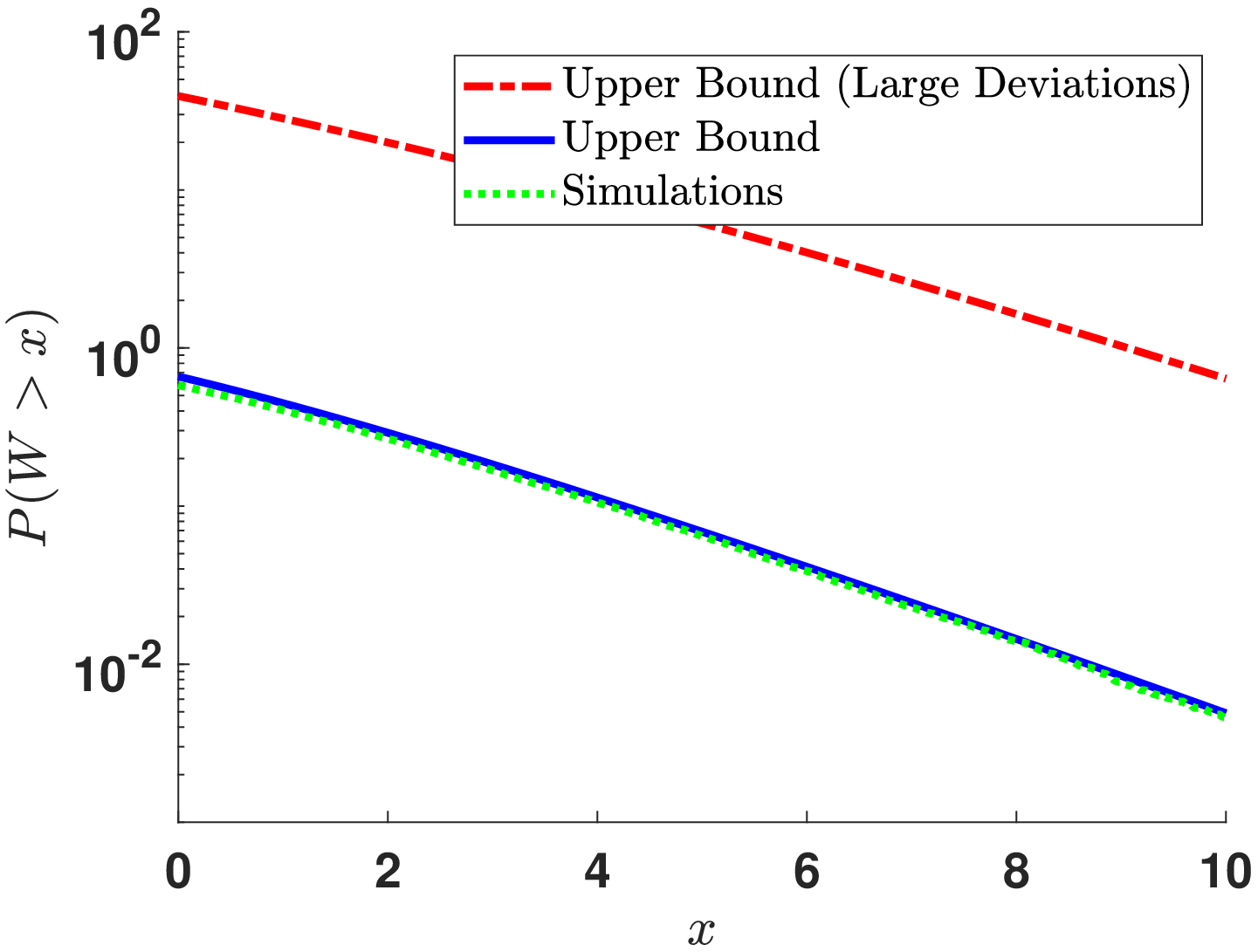}
\\
{\footnotesize (a) $\rho=0.5$} }
\shortstack{\hspace{0cm}
\includegraphics[width=0.323\linewidth,keepaspectratio]{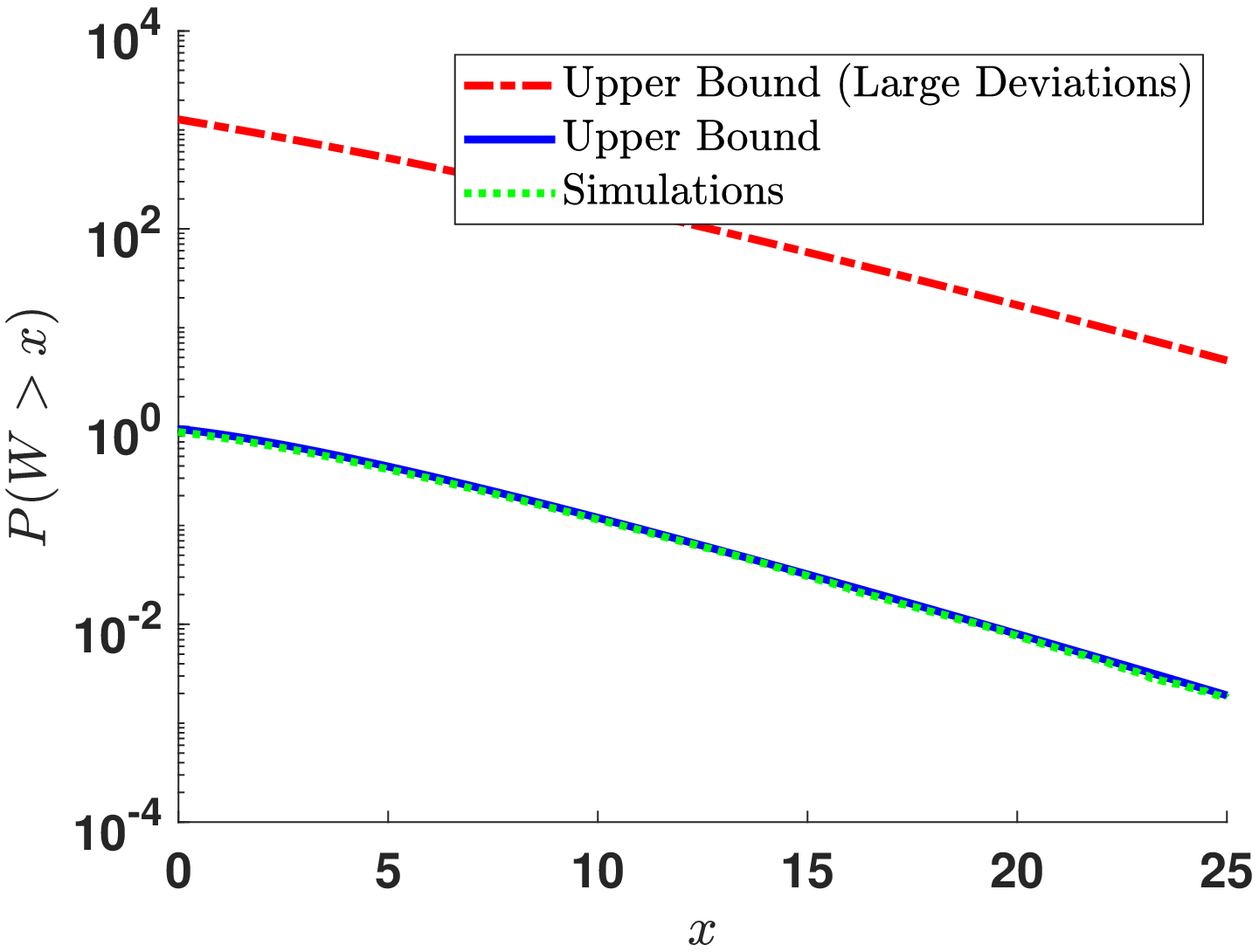}
\\
{\footnotesize (b) $\rho=0.75$}}
\shortstack{\hspace{0cm}
\includegraphics[width=0.323\linewidth,keepaspectratio]{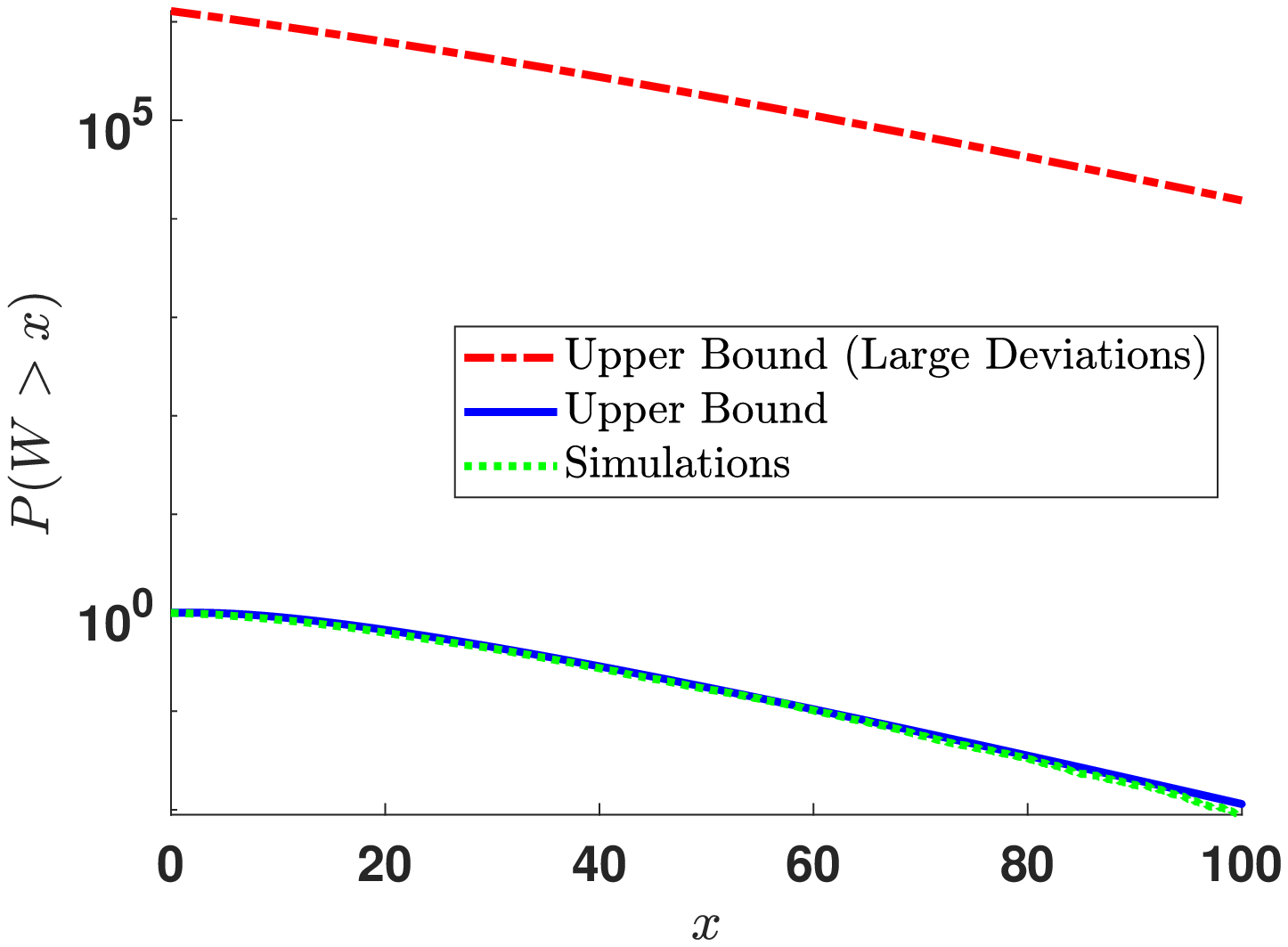}
\\
{\footnotesize (c) $\rho=0.95$} }

 \caption{The waiting time CCDF $\P(W>x)$ for the $Erlang(2)/M/1\rightarrow \cdot/M/1$ tandem: Upper Bounds (from \eqref{eq:dm2_boundcase1}-\eqref{eq:dm2_boundcase2}) vs. Upper Bounds (from \eqref{eq:dm2_ld}, based on large deviations) vs. Simulations} \label{fig:E2M2}
%\end{center}
\end{figure}

Fig.~\ref{fig:DM2} displays the results for the $D/M/1\rightarrow \cdot/M/1$ tandem, for several values of the utilization factor $\rho$; the service rate is $\mu=1$ and the inter-arrival times are $X=\frac{1}{\rho\mu}$. The (extremely) poor accuracy of the large-deviations-based bounds is particularly pronounced at high utilizations, despite the $D/M$ setting; this is due not only to the nested `$\max$', and consequently the nested infinite application of the Union Bound, but also to the more limited choice in running the optimization `$\inf_{\{0<\theta<\mu:\beta<1\}}$', as opposed to scenarios with lower $\rho$. In turn, in addition to the overall accuracy of the bounds from \eqref{eq:dm2_boundcase1}-\eqref{eq:dm2_boundcase2}, we highlight their ability to capture the initial concave `bend' which is clearly visible in (c) around small $x$; this feature is precisely due to the poly-exp structure of the bounds. We note that exponential-only bounds would entirely miss the bend, as they would form a single straight-line on the displayed log-log scale; this is the case for martingale-based bounds developed for single queues (see, e.g.,~\cite{PoCi14}).

Fig.~\ref{fig:E2M2} illustrates results and reveals similar observations about the $Erlang(2)/M/1\rightarrow \cdot/M/1$ tandem, as a special case with Gamma distributed input.

\section{Extensions and Open Problems}\label{sec:extensions}
Here we briefly outline several (practically) important extensions of the key integral representation from \eqref{equationpsi} along with the associated main results.

\subsection{Non-Renewal Arrivals}
The extension to non-renewal arrivals (i.e., non-`$GI$') is immediate. We only present the analysis for an alternate renewal (AR) arrival process, which is sufficient to convey that the further extension to Semi-Markovian arrivals would only additionally involve keeping track of the states of an underlying Markov chain.

Let $X_k$ with $k\in N$ be an AR process driven by two random variables $X^{(1)}$ and $X^{(2)}$, with potentially different distributions, and a Bernoulli r.v. $B$ taking two values $1$ and $2$ with equal probabilities. If $B=1$ then $X_{2k}:=X^{(1)}_k$ and $X_{2k+1}:=X^{(2)}_{k}$; otherwise, if $B=2$, then $X_{2k}:=X^{(2)}_k$ and $X_{2k+1}:=X^{(1)}_{k}$ for $k\geq0$, where $(X^{(1)}_1, X^{(2)}_1)$, $(X^{(1)}_2, X^{(2)}_2),   \ldots $ are i.i.d. copies of $(X^{(1)}, X^{(2)})$. Equivalently, $X_{2k}:=X^{(B)}_k$ and $X_{2k+1}:=X^{(3-B)}_{k}$, i.e., the realization of $B$ determines the initial distribution of $X_0$, whereas those of $X_k$ for $k\geq 1$ alternate.

The solution to the underlying integral equation along with the connection to $\P(W>x)$ are given in~Appendix~\ref{app:AR} for the $M=2$ case, i.e., $AR/G/1\rightarrow\cdot/G/1$; while no result is presented about the existence of poly-exp bounds, we believe the extensions of Theorems~\ref{th:existence2} and~\ref{th:existenceM} to also be immediate.

\subsection{Extension to Packet Networks}
Consider the tandem network from Fig.~\ref{fig:shapem} as a \textit{packet network}. The fundamental difference is that the service time assigned to job (packet) $n$ at the first queue remains the same at all the downstream queues - because packets' sizes are generally not altered by networks. The ``classical" approach to analyze a packet network is based on the so-called ``Kleinrock Independence Assumption" which requires packets' sizes to be independently regenerated at each queue (e.g., the model from Fig.~\ref{fig:shapem}). The original reasoning for this assumption was to fit the packet network analysis within the Jackson networks framework.

While drastically simplifying the analysis, the ``Kleinrock Independence Assumption" lends itself to an incorrect understanding of scaling laws of sojourn times in networks. Concretely, in the case of Poisson arrivals and exponential service times, sojourn times scale super-linearly in the number of queues, more precisely as $\Theta(M\log{M})$ (Vinogradov~\cite{Vinogradov86}). However, once enforcing the ``Kleinrock Independence Assumption", sojourn times only scale linearly as $\Theta(M)$ (as sums of independent exponential random variables).

	\begin{figure}[h]
		\centering
		\includegraphics[width=0.7\linewidth]{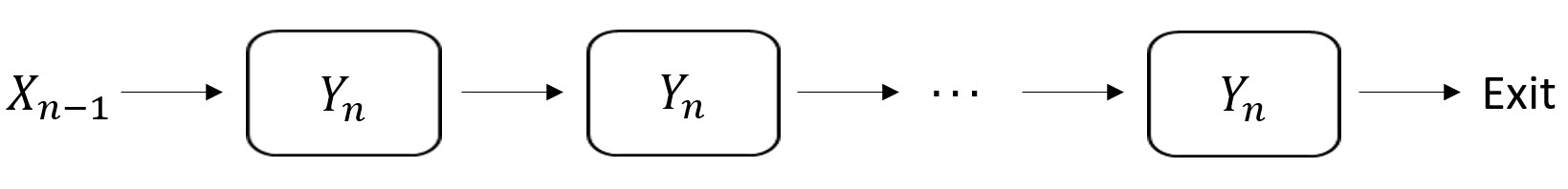}
		\caption{Inter-arrival and (identical) service times for job $n$ in a packet (tandem) network of $M$ queues}
		\label{fig:kleinrock}
	\end{figure}

As illustrated in Fig.~\ref{fig:kleinrock}, the service time of job $n$ is denoted by $Y_n$, i.e., $Y^{(n)}_{i+n-1}=Y_n$ at all the queues $i\in\{1,\dots,M\}$ in the model from Fig.~\ref{fig:shapem}. The exit time of job $n$ can thus be expressed as
\[\begin{split}&\max_{0\leq i_1< \ldots< i_{m}\leq n+m-1}X_0+\cdots+X_{i_1-1}+Y_{i_1}+\cdots+Y_{i_2-1}+Y_{i_2-1}+\cdots+Y_{i_3-2}\\&\qquad\qquad\qquad\qquad+\cdots+Y_{i_{m}-m+1}+\cdots+Y_{n}\\&\qquad=\max_{0\leq i\leq n}\{X_0+\cdots+X_{i-1}+Y_{i}+\cdots+Y_{n}+(M-1)\max\{Y_{i},\ldots, Y_{n}\}\}~.\end{split}\]
The term $(M-1)\max\{Y_{i},\ldots, Y_{n}\}$ attains the maximum over all possible sums, with repetition, from the set $\{Y_{i},\ldots, Y_{n}\}$; there are $n-i+M-1\choose M-1$ such sums.

By changing the sub-indexes $k$ to $n+1-k$ the exit time can be more conveniently written as
\[\begin{split}\tau=\max_{1\leq i\leq  n+1}\{X_{n+1}+\cdots+X_{i+1}+Y_{i}+\cdots+Y_{1}+(M-1)\max\{Y_{i},\ldots, Y_{1}\}\}\end{split}~.\]
Again, the service time of job $n$ becomes $Y_1$, that of job $n-1$ becomes $Y_2$, and so on. It then follows that the waiting time of job $n$ has the same distribution as
\begin{eqnarray}
	W=\tau-\left(MY_1+X_{2}+\cdots +X_{n+1}\right)=\max\left\{{T}^{1}_{2,n+1},{T}^2_{2,n+1}-(M-1)Y_1\right\}\vee 0~,\label{eq:WKleinrock}
\end{eqnarray}
where, with abuse of notation,
\[{T}^1_{k,l}:=\max_{k\leq i\leq l}\left\{(Y_k-X_k)+\ldots +(Y_i-X_i)\right\}\]
\[{T}^2_{k,l}:=\max_{k\leq i\leq l}\left\{(Y_k-X_k)+\ldots +(Y_i-X_i)+(M-1)\max\{Y_k,\ldots , Y_i\}\right\}~.\]

The solution to the underlying integral equation along with the connection to $\P(W>x)$ are given in~Appendix~\ref{app:Kleinrock} for a tandem of $M$ queues. The results have the potential to lend themselves to explicit results for a wide range of arrival processes, especially when combined with an extension to Semi-Markovian arrivals as mentioned earlier; note that unlike in the  general tandem scenario from~\ref{sec:gc}, the expression of $W$ from~\eqref{eq:WKleinrock} is drastically simpler. The state-of-the-art (Le Gall~\cite{LeGall94}) in tandem (packet) networks includes the exact analysis of $\P(D>x)$, in the case of renewal arrivals and general service times, in terms of an integral in the complex domain; explicit solutions (in the real domain) are available in a few special cases including Poisson or deterministic arrivals, and exponential service times; earlier results (e.g., Boxma~\cite{Boxma79} and Vinogradov~\cite{Vinogradov86}) are restricted to Poisson arrivals only.

We lastly discuss some open problems and another potential extension. The key open problem is whether the integral representation from \eqref{equationpsi} admits an exact (and explicit) solution, which would immediately enable an exact and explicit solution for $\P(W>x)$. Progress in this direction may be achieved by firstly investigating the matching lower bounds in Theorem~\ref{th:existence2}, which are likely to also have a poly-exp structure. The provided explicit analysis for $D/M/1\rightarrow \cdot/M/1$ and $Gamma/M/1\rightarrow \cdot/M/1$ is clearly very tedious, despite using a sub-optimal yet numerically effective matching as in Lemma~\ref{cor:gg1gg1exp}; more convenient representations/simplifications along with numerical algorithms may circumvent the tediousness issue, especially in the general $M$ case. While one of the presented extension involves non-renewal arrivals, another immediate one would involve integrating the presented methodology with network calculus concepts. The merge could in principle enable the analysis of tandem networks with scheduling, and in particular the end-to-end waiting/sojourn time of a flow competing at each queue/server with other flows according to some scheduling policy. From a technical point of view, the extension would require replacing the sequences $Y^{(i)}$ with bivariate `service curve processes' or `dynamic F-servers' (in the terminology from Chang~\cite{Book-Chang}, p. 178), which would encode the service received by the competing flow at each server.

\section{Conclusions}
This paper continues relatively recent efforts~\cite{Ross74,Duffield94,Palmowski96,Asmussen03,CiPo18} to improve large-deviations-type of queueing bounds available for broad arrival processes but proverbially very loose. Unlike prior work dedicated to single queues, this paper targets tandem networks of queues which pose extraordinary technical difficulties.

The breakthrough is a new integral equation which suitably encodes the analytical structure of the end-to-end waiting time. Along with an existence result revealing the poly-exp structure of the solution for the  integral equation, (mixture of) poly-exp bounds on $\P(W>x)$ follow by a matching process. The two presented examples are obviously very tedious, even by using a sub-optimal but convenient matching for the polynomials' parameters. They do reveal however the ability of our approach to render sharp bounds not only in heavy traffic, as prior work achieves as well for single queues using variations of Kingman's $GI/G/1$ bound, but also at lower utilizations.

As demonstrated through the two extensions to non-renewal arrivals and packet networks, our approach is seemingly robust and may be applicable to other queueing networks besides tandems.

\bibliographystyle{abbrv}
\bibliography{../../bibliography/stat}

\appendices

\section{A Very-Light Distribution}\label{app:vld}
	Let a random variable $Y$ with density $f_Y(x)=\alpha\1_{x\geq 0}e^{-\mu x}/(1+x^2)$, for a suitable value of $\alpha$, and a random variable $X$ independent of $Y$ and satisfying $X>\E[Y]\vee (\ln (\alpha\pi/2)/\mu) $ a.s. Then
	\[\theta^+:=\sup\{r\in \R: \E[e^{r(Y-X)}]<\infty\}=\mu\]
	and
	\[\E[e^{\theta^+(Y-X)}]=\E[e^{\mu(Y-X)}]=\alpha\E[e^{-\mu X}]\int_0^\infty\frac{1}{1+x^2}dx=\alpha \pi\E[e^{-\mu X}]/2<1~.\]

\section{Semi-Continuous Functions}\label{app:sc}
In our proofs we need two results from analysis. Let $\Dcal$ be a compact subset of $\Rbar^M$, $\phi:\Dcal\to \Rbar$ be an arbitrary function, and define $f:\Dcal\to \Rbar$
	 \[f(x):=\limsup_{y\to x}\phi(y)~\forall x\in \Dcal~.\]

\noindent\textbf{1.}\hspace{0.33cm}The first result is that $f$ is an upper semi-continuous function, i.e., for every convergent sequence $x_n\to x$ in $\Dcal$,
	\[f(x)\geq \limsup_{x_n\to x}f(x_n)~.\]

Indeed, from the definition of $f(x_n)$, there exists  $y_n\in \Dcal$  such that $\lvert y_n-x_n\rvert<n^{-1}$ and $\phi(y_n)\geq f(x_n)-n^{-1}$. Since $x_n\to x$, we get
 $\lvert y_n- x\rvert \leq \lvert y_n-x_n\rvert+\lvert x_n-x\rvert\to 0$ as $n\to \infty$, and hence $y_n\to x$ as well. From $\phi(y_n)\geq f(x_n)-n^{-1}$ it then follows that
 \[f(x)=\limsup_{y\to x}\phi(y)\geq \limsup_{n\to \infty}\phi(y_n)\geq \limsup_{n\to \infty }  f(x_n)~,\]
thus proving that $f$ is upper semi-continuous.

\noindent\textbf{2.}\hspace{0.33cm}The second result is that any upper semi-continuous function on compact domain attains its maximum, according to Weierstrass Maximum Theorem (Borden~\cite{Borden98}, p. 40). In particular, $f$ attains its maximum, i.e., if $K:=\sup_{x\in \Dcal} f(x)$ then
	\[\Kcal:=\{x\in \Dcal: f(x)=K \}\]
	is non-empty and closed. Furthermore,
	\[a_1:=\inf\{x_1\in \Rbar: \exists (x_2,\ldots, x_M)\in \Rbar^{M-1}: (x_1,\ldots, x_M)\in \Kcal\}\]
	\[a_j:=\inf\{x_j\in \Rbar: \exists  (x_j, \ldots, x_M)\in \Rbar^{M-j+1}: (a_1,\ldots,a_{j-1}, x_j, \ldots, x_M)\in \Kcal\}\]
	are well defined and $a:=(a_1,\ldots, a_M)\in \Kcal$.
	Moreover, if $K>0$, then there exists a sequence $x_n\to a$ such that $\phi(x_n)>0$ and $x_n\in \supp(\phi\vee 0)$.

Note the standard notation for a function's support, i.e.,
	\[\supp(\phi):=\{x\in \Dcal: \phi(x)\neq 0\}~\textrm{and}~\supp(\phi\vee 0):=\{x\in \Dcal: \phi(x)>0\}~.\]

\section{Proof of part (c) of Theorem~\ref{th:psigammaeta}}\label{app:3rdpart}
We need to show that  $\psi(u,v)\leq \eta(u,v)$ on $v\geq u$.  Let  $\Dcal:=\{(u,v)\in \Rbar^2, v\geq u\}$ and for $(u,v)\in \Dcal$ define
\[f(u,v):=\limsup_{(x,y)\to (u,v)}(\psi(x,y)-\eta(x,y))~.\]
As in part (b), $f$ is upper-semi continuous and takes its maximum on the compact set $\Dcal$, i.e.,
\[K:=\max_{(u,v)\in \Dcal}f(u,v)~.\]
If $K\leq 0$ the proof is complete. Otherwise, assume that $K>0$ and define
\[\Kcal:=\{(u,v)\in \Dcal: f(u,v)=K\} \]
which is closed subset of $\Rbar^2$, and
\[a:=\min\{u\in \Rbar: \exists v\in \Rbar :(u,v)\in \Kcal \}\]
\[b:=\min\{v\in \Rbar: (a,v)\in \Kcal \}~.\]
Since $f(a,b)=K>0$, there exists a sequence   $(a_n,b_n)\in {\supp (\psi)}\cap \Dcal$, such that $(a_n,b_n)\to (a,b)$ and
 \begin{align*}K&=f(a,b)=\lim_{n\to \infty }(\psi-\eta)(a_n,b_n)\\&\leq \limsup_{n\to \infty}\E\left[\1_{\{a_n\geq  U\}}(\psi-\eta)\left((a_n-U)\wedge (b_n-V), b_n-V\right)\right]\intertext{Since $\psi-\eta\leq 1$, we can apply Fatou's lemma}&\leq \E\left[\limsup_{n\to \infty}\1_{\{a_n\geq  U\}}(\psi-\eta)\left((a_n-U)\wedge (b_n-V), b_n-V\right)\right]\\&\leq K\cdot\P(a=U) +\E\left[\1_{\{a> U\}}f\left((a-U)\wedge (b-V), b-V\right)\right]\\&\leq K\cdot\P(a\geq  U)~,\end{align*}
using the definitions of $K$ and $f$. It then follows that $\P(a\geq U)=1$ such that necessarily
\begin{equation} \label{abinfeta}
	f\left((a-U)\wedge (b-V), b-V\right)=K~\textrm{a.s.}
\end{equation}
Next we claim that $(a,b)=(\infty,\infty)$. Otherwise, if $a<\infty$  then
\eqref{abinfeta} and   $\P(U>  0)>0$ contradict with the choice of $a$ and hence $a=\infty$. Similarly, if $b<\infty$ then \eqref{abinfeta} and $\P(V>0)>0$ contradict with the choice of $b$ and hence $b=\infty$ too.
Finally, \[K=	f(\infty, \infty)=\limsup_{u,v\to \infty}(\psi-\eta)(u,v)=0\]
from the limiting conditions on $\psi$ and $\eta$, and hence $\psi\leq \eta $ on $v\geq u$.

\section{Proofs from \ref{sec:gc}}\label{app:proofsgc}

\begin{proof}{~[of Theorem~\ref{th:psigammaetaM}]} 
 For (a) we can write
		\begin{align*}
			&\psi_M(v_1,\ldots,v_M)=\P(T^{1}_{1}\leq v_1, \ldots, T^{M}_{1}\leq v_M)
			\\&=\P\left( V^{(1)}_1 \leq v_1, V^{(1)}_1+T^1_2\leq v_1, V^{(2)}_1+T^1_2\vee T^2_2\leq v_2,\ldots,  V^{(M)}_1+T^{M-1}_2\vee T^M_2\leq v_M\right)\\&=\P\left( V^{(1)}_1 \leq v_1, T^1_2\leq \bigwedge_{1\leq i\leq 2}\left(v_i-V^{(i)}_1\right), \ldots , T^{M-1}_2\leq \bigwedge_{M-1\leq i\leq M}\left(v_i-V^{(i)}_1\right), T^M_2\leq v_M-V^{(M)}_1  \right)\\
			&= \E\left[\1_{\{v_1\geq  V^{(1)}\}}\psi_M\left(\bigwedge_{1\leq i\leq 2}\left(v_i-V^{(i)}\right),  \ldots,\bigwedge_{M-1\leq i\leq M}\left(v_i-V^{(i)}\right), v_M-V^{(M)} \right)\right]\\&=\E\left[\1_{\{v_1\geq  V^{(1)}\}}\psi_M\left(\Vcal_M(v_1,\ldots, v_M) \right)\right].
		\end{align*}
		For (b) and (c) denote
		\[\phi_1:=\gamma_M-\psi_M, \quad \phi_2:=\psi_M-\eta_M\]
		and
		\[f_i(v_1,\ldots,v_M):=\limsup_{(u_1,\ldots,u_M)\to (v_1,\ldots, v_M)}\phi_i(u_1,\ldots,u_M)~,\]
		which are upper semi-continuous on $\Rbar^M$ and attain their maximum
		\[K_i:=\max_{(v_1,\ldots, v_M)\in \Rbar^M}f_i(v_1,\ldots, v_M), \qquad i=1,2.\]
		Since $f_i(\infty,\ldots,\infty)=0$ it follows that $K_i\geq 0$. Let
		\[\mathcal{K}_i:=\{(v_1,\ldots,v_M)\in \Rbar^M: f_i(v_1,\ldots, v_M)=K_i\}\]
		and define
		\[a_1^{(i)}:=\min\{v_1\in \Rbar: \exists (v_2,\ldots,v_M)\in \Rbar ^{M-1}: (v_1,\ldots,v_M)\in  \Kcal_i\}\]
		\[a_j^{(i)}:=\min\{v_j\in \Rbar: \exists (v_{j+1},\ldots ,v_M)\in \Rbar^{M-j} ,  (a_1^{(i)},\ldots, a_{j-1}^{(i)},v_j,\ldots,v_M)\in \Kcal_i\}~,\]
		which are well-defined since $f_i$ is upper semi-continuous; also, \[(a_1^{(1)},\ldots,a_M^{(1)})\in \mathcal{K}_1\subseteq \supp(\gamma_M \vee 0), \quad (a_1^{(2)},\ldots,a_M^{(2)})\in \mathcal{K}_2\subseteq \supp(\psi_M).\] We can now write
		\[ \begin{aligned}K_i&=f_i(a_1^{(i)},\ldots,a_M^{(i)})\\&=\limsup_{(v_1,\ldots, v_M)\to (a_1^{(i)},\ldots,a_M^{(i)})}\phi_i(v_1,\ldots, v_M)\\&=\limsup_{(v_1,\ldots, v_M)\to (a_1^{(i)},\ldots,a_M^{(i)})} \E\left[\1_{\{v_1\geq V^{(1)}\}}\phi_i(\Vcal_M(v_1,\ldots, v_M))\right]\\&\leq  \E\left[\limsup_{(v_1,\ldots, v_M)\to (a_1^{(i)},\ldots,a_M^{(i)})} \1_{\{v_1\geq V^{(1)}\}}\phi_i(\Vcal_M(v_1,\ldots, v_M))\right]\\&\leq K_i\cdot \P(a_1^{(i)}= V^{(1)})+ \E\left[ \1_{\{a_1^{(i)}> V^{(1)}\}}f_i(\Vcal_M(a_1^{(i)},\ldots, a_M^{(i)}))\right]\\&\leq K_i\cdot \P(a_1^{(i)}\geq  V^{(1)})\end{aligned}\]
		Hence $\P(a_1^{(i)}\geq  V^{(1)})=1$ and
		\begin{equation} \label{a_inf}
			\begin{split}
				&f_i(\Vcal_M(a_1^{(i)},\ldots, a_M^{(i)}))\\&=f_i\left(\bigwedge_{1\leq j\leq 2}\left(a_j^{(i)}-V^{(j)}\right),  \ldots,\bigwedge_{M-1\leq j\leq M}\left(a_j^{(i)}-V^{(j)}\right), a_M^{(i)}-V^{(M)} \right)=K_i~\textrm{a.s.}
			\end{split}
		\end{equation}
		We now prove by contradiction that $(a_1^{(i)},\ldots,a_M^{(i)})=(\infty,\ldots,\infty)$. Letting  \[j_0: =\min\{j\in\{1,2,\ldots,m\}, a_j^{(i)}<\infty\}\]
		it then follows from \eqref{a_inf} and the choice of $a_{j_0}$ that $\P(V^{(j_0)}>  0)=0 $, thus contradicting the assumption that $\P(V^{(j_0)}>  0)>0 $.
		Therefore \[K_i=	f_i(\infty,\ldots,\infty)=0~,\]
		and hence $\psi_M\leq \gamma_M$ and $\eta_M\leq \psi_M$.
		
		We can now prove the uniqueness of $\psi$ solving for (\ref{psimnode}). Let $\psi_1$ and $\psi_2$ be two bounded solutions satisfying $\psi_i(\infty,\ldots, \infty)=\lim_{v_1,\dots,v_M\to \infty}\psi_i(v_1,\dots,v_M)=1$. Applying the second part of the theorem with $\psi=\psi_i$ and $\gamma=\psi_{3-i}$ (note that the proof only needs that $\psi$ satisfies (\ref{psimnode}), is bounded, and $\psi(\infty,\ldots, \infty)=\lim_{v_1,\dots,v_M\to \infty}\psi(v_1,\dots,v_M)=1$) we obtain that $\psi_i\geq \psi_{3-i}$ for $i=1,2$, and hence $\psi_1=\psi_2$.

			\end{proof}

\begin{proof}{~[of Corollary~\ref{cor:WgammaetaM}]}
			We have for $x\geq 0$
			\[\begin{split}\P(W> x)&=\P\left(\max_{1\leq j\leq M}\{T^j_{M+1}+S_j\}> x+\sum_{k=1}^M Y^{M+1-k}_k\right)\\&=1-\P\left(\max_{1\leq j\leq M}\{T^j_{M+1}+S_j\}\leq x+\sum_{k=1}^M Y^{M+1-k}_k\right)\\&=1-\P\left(\forall {1\leq j\leq M}, T^j_{M+1}\leq x+\sum_{k=1}^M Y^{M+1-k}_k-S_j\right)\\&=1-\E\left[\psi_M\left(x+\sum_{k=1}^M Y^{M+1-k}_k-S_1, \ldots,x+\sum_{k=1}^M Y^{M+1-k}_k-S_M \right)\right]\end{split}~,\]
from the stationarity of $T_k^j$ for all $k\in\N$.
\end{proof}

	\begin{proof}{~[of Theorem~\ref{th:existenceM}]} As in the case $M=2$ (Theorem~\ref{th:existence2}) we proceed in two steps.

		\textit{Step 1: } First we prove by induction on $i\geq1$ that there exist the polynomials $Q_{i}:\R\to \R, i\in\{1,2,\ldots,m\}$ with degrees at most
		$	d_{i}(\theta_i)$, respectively, having non-negative values on $[0,\infty) $, such that for all $ v\geq -a_+$
		\begin{equation}\label{Q1psi}
			Q_1e^{-\theta_1 v}\geq \E\left[\1_{\{v\geq V^{(1)}\}}Q_1e^{\theta_1(V^{(1)}-v)}\right]+\P(v< V^{(1)})
		\end{equation}
		and for all $2\leq i\leq M$ and $v\geq 0$
		\begin{equation}\label{Qmpsi}
			Q_{i}(v)e^{-\theta_i v+\theta_i a}\geq \E\left[\1_{\{v\geq V^{(i)}\}}\sum_{l=i-1}^{i}Q_{l}\left(v-V^{(i)}\right)e^{\theta_l(V^{(i)}-v+a)}\right]+\P(v<V^{(i)})~.
		\end{equation}

		\textit{Proof of Step 1: }
		The case $i=1$ follows immediately as in the proof of Theorem~\ref{th:existence2} by letting
		\[Q_{1}:=\left(\inf_{v\geq -a_+}\E\left[e^{\theta_1 (V^{(1)}-v)}\mid V^{(1)}>v\right]\right)^{-1}~.\]
For some $k\geq2$ we next assume the existence of the polynomials $Q_i$ for $i\leq k-1$. We need to prove that there exists $Q_k(v):=\sum_{j=0}^{d_{k}(\theta_k)} A_{j}  v^{j}$ such that \eqref{Qmpsi}  holds for $i=k$.
		It is thus sufficient to show that there exists the non-negative constants $A_{d_{k}(\theta_k)}, \ldots , A_0$ such that
		\begin{equation}\label{Al-psi-A0M}
			\begin{split}
				&\sum_{j=0}^{d_{k}(\theta_k)} A_{j}e^{\theta_k a} \left\lbrace v^{j}e^{-\theta_k v}-\E\left[\1_{\{v\geq V^{(k)}\}}\left(v-V^{(k)}\right)^{j}e^{\theta_k (V^{(k)}-v)} \right]\right\rbrace \\&\qquad\geq \E\left[\1_{\{v\geq V^{(k)}\}}Q_{k-1}\left(v-V^{(k)}\right)e^{\theta_{k-1} (V^{(k)}-v+a)}\right]+\P(v<V^{(k)})~.
			\end{split}
		\end{equation}
Next we bound both side and then show the existence of the $A_j$'s for the tighter inequality. An upper bound on the right side is
		\begin{align*}
			&\E\left[\1_{\{v\geq V^{(k)}\}}Q_{k-1}\left(v-V^{(k)}\right)e^{\theta_{k-1} (V^{(k)}-v+a)}\right]+\P(v<V^{(k)})\\&\leq e^{\theta_{k-1}a} \E\left[\1_{\{v\geq V^{(k)}\}}Q_{k-1}\left(v-V^{(k)}\right)e^{\zeta_k (V^{(k)}-v)}\right]+\P(v<V^{(k)})\\&=  e^{\theta_{k-1}a}\E\left[Q_{k-1}\left(v-V^{(k)}\right)e^{\zeta_k (V^{(k)}-v)}\right]+\P(v<V^{(k)})\\&\quad - e^{\theta_{k-1}a}\E\left[Q_{k-1}\left(v-V^{(k)}\right)e^{\zeta_k (V^{(k)}-v)}\mid  V^{(k)}>v\right]\P(V^{(k)}> v)\\&\leq R_{d_{k-1}(\theta_{k-1})}(v)e^{-\zeta_k v} +\left(1-\inf_{v\geq 0} e^{\theta_{k-1}a}\E\left[Q_{k-1}\left(v-V^{(k)}\right)e^{\zeta_k (V^{(k)}-v)}\mid V^{(k)}>v\right]\right)\P(V^{(k)}>v)\\&\leq R_{d_{k-1}(\theta_{k-1})}(v)e^{-\zeta_k v}+\left(1+C \sum_{l=0}^{\left\lfloor (d_{k-1}(\theta_{k-1})-1)/2\right\rfloor}K_k^{2l+1}\right)\P(V^{(k)}>v)~,
		\end{align*}
		where $R_{d_{k-1}(\theta_{k-1})}$ is a polynomial of degree at most $d_{k-1}(\theta_{k-1})$ and $C$ is a positive constant. In the last inequality we used the induction hypothesis on $Q_{k-1}$ and only bounded the odd powers of $(v-V^{(k)})$ using the assumption on the conditional expectations; the other terms are positive.
		
Next we lower bound the terms in brackets from the left side of \eqref{Al-psi-A0M}. For $j\geq 1$
		\begin{align*}
			&v^je^{-\theta_k v}-\E\left[\1_{\{v\geq  V^{(k)}\}}\left(v-V^{(k)}\right)^je^{\theta_k(V^{(k)}-v)} \right]\\&=v^je^{-\theta_k v}-\E\left[\left(v-V^{(k)}\right)^je^{\theta_k(V^{(k)}-v)}\right] +\E\left[\1_{\{V^{(k)}>v\}}\left(v-V^{(k)}\right)^je^{\theta_k(V^{(k)}-v)}\right]\\&\geq \left(1-\E\left[e^{\theta_k V^{(k)}}\right]\right)v^je^{-\theta_k v}+j \E\left[V^{(k)}e^{\theta_k V^{(k)}}\right]v^{j-1}e^{-\theta_k v} +\tilde{R}_{j-2}(v)e^{-\theta_k v} \\&\quad+\inf_{v\geq 0}\E\left[\left(v-V^{(k)}\right)^je^{\theta_k(V^{(k)}-v)} \mid V^{(k)}>v\right]\P(V^{(k)}>v)\\&\geq \left(1-\E\left[e^{\theta_k V^{(k)}}\right]\right)v^je^{-\theta_k v}+j\E\left[V^{(k)}e^{\theta_k V^{(k)}}\right]v^{j-1}e^{-\theta_k v}\\&\quad +\tilde{R}_{j-2}(v)e^{-\theta_k v} -K_k^j\1_{\{j\in 2\Z+1\}}\P(V^{(k)}>v)~,
		\end{align*}
		where $\tilde{R}_{j-2}$ is a polynomial of degree at most $j-2$; in the last inequality we only bounded the conditional expectation with $K_k^j$ when $j$ is odd, by accounting for the underlying sign. Also, for $j=0$,
		\begin{align*}
			&e^{-\theta_k v}-\E\left[\1_{\{v\geq  V^{(k)}\}}e^{\theta_k (V^{(k)}- v)} \right]\\& \geq \left(1-\E\left[e^{\theta_k V^{(k)}}\right]\right)e^{-\theta_k v}+\inf_{v\geq 0}\E\left[e^{\theta_k (V^{(k)}- v)}\mid  V^{(k)}>v\right]\P(V^{(k)}>v)\\&\geq \left(1-\E\left[e^{\theta_k V^{(k)}}\right]\right)e^{-\theta_k v}+\P( V^{(k)}>v)~.
		\end{align*}

It is thus sufficient to find the coefficients $A_j$'s satisfying the tighter inequality:
		\begin{equation}\label{Al-psi-A0M-tight}
			\begin{split}
				&A_{0}e^{\theta_k a} \left\lbrace \left(1-\E\left[e^{\theta_k V^{(k)}}\right]\right)e^{-\theta_k v}+\P( V^{(k)}>v)\right\rbrace\\
&\qquad+\sum_{j=1}^{d_{k}(\theta_k)} A_{j}e^{\theta_k a} \Bigg\lbrace \left(1-\E\left[e^{\theta_k V^{(k)}}\right]\right)v^je^{-\theta_k v}+j\E\left[V^{(k)}e^{\theta_k V^{(k)}}\right]v^{j-1}e^{-\theta_k v}\\&\qquad\qquad\qquad +\tilde{R}_{j-2}(v)e^{-\theta_k v} -K_k^j\1_{\{j\in 2\Z+1\}}\P(V^{(k)}>v)\Bigg\rbrace \\&\qquad\geq R_{d_{k-1}(\theta_{k-1})}(v)e^{-\zeta_k v}+\left(1+C \sum_{l=0}^{\left\lfloor (d_{k-1}(\theta_{k-1})-1)/2\right\rfloor}K_k^{2l+1}\right)\P(V^{(k)}>v)~.
			\end{split}
		\end{equation}

There are two main cases:

If $\theta_k<\theta_{k-1}$ (Case 1), i.e., node $k$ is the first bottleneck in the sequence $(1,2,\dots,k)$, then $d_k(\theta_k)=I_k(\theta_k)\in\{0,1\}$\footnote{In this case, $I_k(\theta_k)=0$ happens when $\E[e^{\theta (Y^{(k)}-X)}]< 1~\forall \theta>0$; recall the discussion on light-tailed distributions from the beginning of \ref{sec:m2}.}. If $\E[e^{\theta_kV^{(k)}}]<1$ (Case 1.1) then we first choose $A_1=0$. If $d_{k-1}(\theta_{k-1})=0$ (Case 1.1.1) then it is easy to see that $A_0>0$ sufficiently large is sufficient to satisfy \eqref{Al-psi-A0M-tight}. Instead, if $d_{k-1}(\theta_{k-1})>0$ (Case 1.1.2) then we recall that, by definition, $\zeta_k>\theta_k$; for this very reason, even if the polynomial $R_{d_{k-1}(\theta_{k-1})}(v)$ on the right side has a higher degree than $0$, a sufficiently large $A_0>0$ is sufficient to satisfy \eqref{Al-psi-A0M-tight}.

If $\E[e^{\theta_kV^{(k)}}]=1$ (Case 1.2) then $d_k(\theta_k)=1$. If $d_{k-1}(\theta_{k-1})=0$ (Case 1.2.1) then the existence of $A_0$ and $A_1$ is obvious. Otherwise, if $d_{k-1}(\theta_{k-1})\geq1$ (Case 1.2.2) then, again, $\zeta_k>\theta_k$, and the existence of $A_0$ and $A_1$ follows as in Case 1.1.2.

En passant, we point out that in the case $M=2$ from Theorem~\ref{th:existence2}, $d_1(\theta)=0~\forall\theta$ by definition and hence the polynomial on the right side has always a degree of $0$; for this reason, the additional parameter $\zeta$ along with the corresponding condition on conditional expectation were not necessary.

In the other main case, i.e., $\theta_k=\theta_{k-1}$ (Case 2), the terms containing $\P(V^{(k)}>v)$ are properly bounded by choosing a sufficiently large $A_0>0$. If  $\E[e^{\theta_kV^{(k)}}]<1$ (Case 2.1) then the coefficient of $A_j$ is a polynomial of degree $j$ with positive dominant coefficient, and since the right side of \eqref{Al-psi-A0M-tight} has degree $d_{k-1}(\theta_{k-1})=d_k(\theta_k)$ (because $\E[e^{\theta_kV^{(k)}}]<1\Rightarrow I_k(\theta_k)=0 $) there exist non-negative  $A_{d_k}, \ldots, A_0$ satisfying \eqref{Al-psi-A0M-tight}. Similarly,  if $\E[e^{\theta_kV^{(k)}}]=1$, the coefficient of $A_j$ is a polynomial of degree $j-1$ with positive dominant coefficient  and since the right side of \eqref{Al-psi-A0M-tight} has degree $d_{k-1}(\theta_{k-1})=d_k(\theta_k)-1$ (because $\E[e^{\theta_kV^{(k)}}]=1\Rightarrow I_k(\theta_k)=1$) it follows that there exist non-negative  $A_{d_k}, \ldots, A_0$ satisfying \eqref{Al-psi-A0M-tight}. Step 1 is thus complete.
		
		\textit{Step 2: }Let the $Q_{j}$'s from Step 1 and define $P_{j}(v_1,\ldots,v_j):=Q_{j}(v_j+a)~\forall j\in \{1,2,\ldots,m\}$. Then $\gamma_M$ satisfies \eqref{gammamnode}.
		
		\textit{Proof of Step 2: } We prove by induction on $m=1,\dots,M$ that the marginal functions
 $\gamma_m$ of $\gamma_M$ (i.e., restricted to the first $m$ components) satisfy \eqref{gammamnode}\footnote{With abuse of notation, in Step 2, when referring for some $m$ to \eqref{gammamnode} or other expressions in $M$, we mean the corresponding restrictions as if $M=m$.}. For $m=1$,
		\[\gamma_{1}(v_1):=(1-Q_{1}e^{-\theta_1 v_1})\1_{\{v_1\geq -a_+\}}\]
		satisfies for all $v_1\geq -a_+$
		\begin{align*}
			\E\left[\1_{\{v_1\geq V^{(1)}\}}\gamma_{1}(v_1-V^{(1)})\right]&=\E\left[\1_{\{v_1\geq V^{(1)}\}}\left(1-Q_{1}e^{-\theta_1 v_1+\theta_1 V^{(1)}}\right)\right]\\&\geq 1- Q_{1}e^{-\theta_1 v_1}=\gamma_{1}(v_1)~,
		\end{align*}
according to the construction from \eqref{Q1psi}.

For the induction step we assume that the statement holds for `$m-1$' and we need to prove it for `$m$'. Because $Q_{j}$'s
		satisfy Step 1 for $j=1,\dots,m-1$ and the random variables $V^{(1)},\ldots, V^{(m-1)}$, we have that
		\[ \gamma_{m-1}(v_1,\ldots,v_{m-1}):=\1_{\{(v_1,\ldots,v_{m-1})\in \Dcal^{m-1}_a\}}\left[ 1-\sum_{j=1}^{m-1}Q_j( v_{j}+a)e^{-\theta_j v_j}\right]\]
		satisfies \eqref{gammamnode}. Noting that
		\[\supp(\gamma_m\vee 0):=\{(v_1,\ldots, v_m)\in \Dcal^m_a: 1> \sum_{j=1}^{m}Q_{j}(v_j+a)e^{-\theta_j v_j} \} \]
it follows that if $(v_1,\ldots, v_m)\in \supp(\gamma_m\vee 0)$ then $(v_1,\ldots, v_{m-1})\in \supp(\gamma_{m-1}\vee 0)$.
		For all $(v_1,\ldots, v_m)\in \supp(\gamma_m\vee 0)$,
		\begin{align*}
			&\E\left[\1_{\{v_1\geq  V^{(1)}\}}\gamma_m\left(  \bigwedge_{1\leq i\leq 2}\left(v_i-V^{(i)}\right),  \ldots,\bigwedge_{m-1\leq i\leq m}\left(v_i-V^{(i)}\right), v_m-V^{(m)} \right)\right]\\&=\E\Bigg[\1_{\{v_1\geq  V^{(1)}, \forall 2\leq i\leq m:  v_i+a\geq V^{(i)}\}} \Bigg\lbrace 1-\sum_{j=1}^{m-1} Q_{j}\left(\bigwedge_{j\leq i\leq j+1}\left(a+v_i-V^{(i)}\right)\right)e^{-\theta_j\left(\bigwedge_{j\leq i\leq j+1}\left(v_i-V^{(i)}\right)\right)}\\&\qquad -Q_m\left(a+v_m-V^{(m)}\right)e^{-\theta_m\left(v_m-V^{(m)}\right)}\Bigg\rbrace\Bigg]\intertext{Since $Q_{m-1}(x)\geq 0~\forall x\geq0$, by construction, we have}&\geq \E\Bigg[\1_{\{v_1\geq  V^{(1)}, \forall 2\leq i\leq m:  v_i+a\geq V^{(i)}\}} \\&\qquad \Bigg\lbrace 1-\sum_{j=1}^{m-1} Q_{j}\left(\bigwedge_{j \leq i\leq (j+1)\wedge (m-1)}\left(a+v_i-V^{(i)}\right)\right)e^{-\theta_j\left(\bigwedge_{j \leq i\leq (j+1)\wedge (m-1)}\left(v_i-V^{(i)}\right)\right)}\\&\qquad -\sum_{j=m-1}^mQ_{j}\left(a+v_m-V^{(m)}\right)e^{-\theta_j\left(v_m-V^{(m)}\right)}\Bigg\rbrace \Bigg]\\&=	\E\Biggl[\1_{\{v_1\geq  V^{(1)}, v_m+a<V^{(m)}, \forall 2\leq i\leq m-1,  v_i+a\geq V^{(i)} \}} \\&\qquad \Bigg\lbrace -1+\sum_{j=1}^{m-1} Q_{j}\left(\bigwedge_{j \leq i\leq (j+1)\wedge (m-1)}\left(a+v_i-V^{(i)}\right)\right)e^{-\theta_j\left(\bigwedge_{j \leq i\leq (j+1)\wedge (m-1)}\left(v_i-V^{(i)}\right)\right)}\Bigg\rbrace\Biggr]\end{align*}\begin{align*}&\quad +	\E\Bigg[\1_{\{v_1\geq  V^{(1)},  \forall 2\leq i\leq m-1,  v_i+a\geq V^{(i)}\}}\\&\qquad \left\lbrace 1-\sum_{j=1}^{m-1} Q_{j}\left(\bigwedge_{j \leq i\leq (j+1)\wedge (m-1)}\left(a+v_i-V^{(i)}\right)\right)e^{-\theta_j\left(\bigwedge_{j \leq i\leq (j+1)\wedge (m-1)}\left(v_i-V^{(i)}\right)\right)}\right\rbrace \Bigg]\\&\quad -\E\left[\1_{\{v_m+a\geq  V^{(m)}\}}\sum_{j=m-1}^mQ_{j}\left(a+v_m-V^{(m)}\right)e^{-\theta_j\left(v_m-V^{(m)}\right)}\right]\intertext{Using the positivity of the $Q_j$'s for the first expectation and the induction hypothesis for the second we can continue}&\geq  -\P\left(v_m+a< V^{(m)}\right)+\left(1-\sum_{j=1}^{m-1}Q_{j}(a+v_j)e^{-\theta_jv_j}\right)\\&\quad -\E\left[\1_{\{v_m+a\geq  V^{(m)}\}}\sum_{j=m-1}^mQ_{j}\left(a+v_m-V^{(m)}\right)e^{-\theta_j\left(v_m-V^{(m)}\right)}\right]
			\\&\geq 1-\sum_{j=1}^{m}Q_{j}(a+v_j)e^{-\theta_j v_j}=\gamma_m(v_1,\ldots, v_m)~.
		\end{align*}
In the last inequality we applied Step 1 on the event $v_m+a\geq V^{(m)}$.

Finally, the remaining conditions from part (b) of Theorem~\ref{th:psigammaeta}, i.e., $\gamma$ is bounded and $\gamma(\infty,\infty)=1$ hold trivially.
	\end{proof}

\section{Finding the Polynomials' Parameters for $GI/G/1\rightarrow \cdot/G/1$}\label{app:fppGM}
\begin{lemma}\label{cor:gg1gg1exp}
The following set of eight inequalities is sufficient to satisfy \eqref{eq:gg1gg1cond}
		\[{\color{blue}A\E\left[e^{\theta V}\mid U> u\right]+B\E\left[(v-V)e^{\theta V}\mid U> u\right]+C\E\left[(u-U)e^{\theta V}\mid U> u\right]\geq 0}\]
		\[{\color{orange}B\E\left[Ve^{\theta V}\right]+C\E\left[Ue^{\theta V}\right]\geq 0}\]
		\[{\color{green}(A+D)e^{\theta a}K_0^Z(v+a+X)-(B+C)e^{\theta a}\left(a K_0^Z(v+a+X)+K_1^Z(v+a+X)\right)\geq 1}\]
		\[{\color{purple}CK_1^Z(v-u+Y)+D(1-K_0^Z(v-u+Y))\geq 0}\]
		\[{\color{violet}D\E[K_0^Y(u+X)\1_{\{Y>u+X\geq 0\}}]\geq \P(Y>u+X\geq 0) }\]\[{\color{brown}D\E[e^{-\theta (u+X)}\mid0>u+X]\P(Y=0)\geq \P(Y=0)}\]\[ {\color{pink}DK_0^Y(0)\E[e^{-\theta (u+X)}\mid 0>u+X]\geq 1}\]
		{\color{red}\[\begin{aligned}&Ae^{\theta a}\E [K_0^Z(v+a+X)\1_{\{Z> v+a+X, u+X+a\wedge 0<  Y \leq  u+X\}}]\\&-Be^{\theta a}\Big(\E [K_1^Z(v+a+X)\1_{\{Z> v+a+X, u+X+a\wedge 0< Y \leq  u+X\}}]\\&\qquad+a\E [K_0^Z(v+a+X)\1_{\{Z> v+a+X, u+X+a\wedge 0< Y \leq  u+X\}}]\Big)\\&+Ce^{\theta a}\E[(u-U)K_0^Z(v+a+X)\1_{\{0\geq U-u>a\wedge 0, V>v+a\}}]+D\E\left[\1_{\{ u\geq U> a\wedge 0+u, V>a+v\}}e^{-\theta u+\theta U}\right]\\&\geq \E[\1_{\{Z> v+a+X, u+X+a\wedge 0<  Y \leq  u+X\}}]~.\end{aligned}\]}
\end{lemma}
When $a\geq0$, the last inequality is satisfied automatically as all the indicators vanish; also, the sixth inequality is satisfied trivially when $Y$ is a continuous r.v. The eight inequalities are obtained by expanding \eqref{eq:gg1gg1cond} and grouping terms from its left and right sides. Due to the obvious tediousness we depict each group in a different colour. As already hinted, the above grouping is likely sub-optimal, i.e., different groupings or other more direct approaches not involving simplifying groups may lend themselves to better $a,A,B,C,D$ in the sense of minimizing the bound on $\P(W>x)$.

\begin{proof}
Inequality \eqref{eq:gg1gg1cond} is equivalent to
		\begin{equation} \label{ABCDpsi1}
			AF_A+BF_B+CF_C+DF_D\geq F_1
		\end{equation}
		where for $u\geq -a_+, v\geq -a, v\geq u$,
		\begin{align*}
			F_A(u,v)&:=e^{-\theta v} -\E\left[\1_{\{u\geq  U, v+a\geq  V\}}e^{\theta (V-v)} \right]\\&\ =\E\left[\1_{\{u< U\}}e^{\theta (V-v)} \right]+\E\left[\1_{\{u\geq U, v+a<  V\}}e^{\theta (V-v)} \right]\\&\ ={\color{blue}e^{-\theta v}\E\left[e^{\theta V}\mid U> u\right]\P(U> u)}\\&\ \quad {\color{green} +e^{\theta a}\E [K_0^Z(v+a+X)\1_{\{Z> v+a+X, Y \leq  u+X+a\wedge 0\}}]}\\&\ \quad {\color{red} +e^{\theta a}\E [K_0^Z(v+a+X)\1_{\{Z> v+a+X, u+X+a\wedge 0<  Y \leq  u+X\}}]}~.
\end{align*}
The last two lines are obtained from the second term on the second line by using standard properties of conditional expectation and the independence of $X$, $Y$, $Z$:
\begin{align*}
\E\left[\1_{\{u\geq U, v+a<  V\}}e^{\theta (V-v)} \right]&=e^{\theta a}\E\left[\E\left[\1_{\{Z>X+v+a\}}\1_{\{Y-X\leq u\}}e^{\theta(Z-X-v-a)}\mid X\right]\right]\\
&=e^{\theta a}\E\left[\E\left[\1_{\{Z>X+v+a\}}e^{\theta(Z-X-v-a)}\mid X\right]\E\left[\1_{\{Y-X\leq u\}}\mid X\right]\right]\\
&=e^{\theta a}\E\left[\E\left[e^{\theta(Z-X-v-a)}\mid Z>X+v+a, X\right]\P(Z>X+v+a\mid X)\E\left[\1_{\{Y-X\leq u\}}\mid X\right]\right]\\
&=e^{\theta a}\E\left[K_0^Z(X+v+a)\P(Z>X+v+a\mid X)\E\left[\1_{\{Y-X\leq u\}}\mid X\right]\right]\\
&=e^{\theta a}\E\left[\E\left[K_0^Z(X+v+a) \1_{\{Z>X+v+a\}}\1_{\{Y-X\leq u\}}\mid X\right]\right]\\
&=e^{\theta a}\E\left[K_0^Z(X+v+a) \1_{\{Z>X+v+a\}}\1_{\{Y-X\leq u\}}\right]~.
\end{align*}
In the next to last line we used the measurability of $K_0^Z(X+v+a)$ and $\P(Z>X+v+a\mid X)$ with respect to the $\sigma$-field generated by $X$. The indicator $\1_{\{Y-X\leq u\}}$ easily expands to obtain the last two lines from the previous equation (for the formation of the convenient `coloured' groups). We will use the same argument in the expansions of $F_B$, $F_C$, $F_D$:

\begin{align*}
			F_B(u,v)&:=ve^{-\theta v} -\E\left[\1_{\{u\geq U, v+a\geq V\}}(v- V)e^{\theta (V-v)} \right]
			\\&\ =\E\left[\1_{\{U> u \text{ or }   V>v+a \}}(v-V)e^{-\theta v+\theta V}\right]+
			\E\left[Ve^{\theta (V-v)} \right]
			\\&\ = {\color{blue}e^{-\theta v}\E\left[(v-V)e^{\theta V}\mid U> u\right]\P(U> u)} {\color{orange}+\E\left[Ve^{\theta V}\right]e^{-\theta v}}\\&\ \quad{\color{green}-e^{\theta a}\E [K_1^Z(v+a+X)\1_{\{Z> v+a+X, Y \leq  u+X+a\wedge 0\}}]}\\&\ \quad {\color{green}-ae^{\theta a}\E [K_0^Z(v+a+X)\1_{\{Z> v+a+X, Y \leq  u+X+a\wedge 0\}}]}\\&\ \quad{\color{red}-e^{\theta a}\E [K_1^Z(v+a+X)\1_{\{Z> v+a+X, u+X+a\wedge 0< Y \leq  u+X\}}]}\\&\ \quad {\color{red}-ae^{\theta a}\E [K_0^Z(v+a+X)\1_{\{Z> v+a+X, u+X+a\wedge 0< Y \leq  u+X\}}]}
			\intertext{$ \ $}
			F_C(u,v)&:=ue^{-\theta v} -\E\left[\1_{\{ u\geq  U, v+a\geq   V\}}\left( (u- U) \wedge(v- V) \right) e^{\theta (V-v)} \right]\\&\
			=\E\left[\1_{\{u< U \text{ or } v+a<  V\}}(u-U)e^{-\theta v+\theta V}\right]+
			\E\left[Ue^{\theta (V-v)} \right]\\&\ \quad + \E\left[\1_{\{ u\geq  U, a\geq   V-v\geq U-u\}}\left( u-v-U+V \right) e^{\theta (V-v)} \right] \\&\
			=\E\left[\1_{\{u< U\}}(u-U)e^{-\theta v+\theta V}\right]+\E\left[\1_{\{u+a\wedge0\geq  U,  V-v>a\}}(u-U)e^{-\theta v+\theta V}\right]\\&\ \quad +\E\left[\1_{\{0\geq  U-u>a\wedge0,  V-v>a\}}(u-U)e^{-\theta v+\theta V}\right]
			\\&\ \quad+\E\left[Ue^{\theta (V-v)} \right]-\E\left[\1_{\{ a\wedge 0+u\geq  U,    V-v>a\}}\left( u-v-U+V \right) e^{\theta (V-v)} \right] \\&\ \quad+ \E\left[\1_{\{ a\wedge 0+u\geq  U,V-v\geq U-u\}}\left( u-v-U+V \right) e^{\theta (V-v)} \right]\\&\
			=\E\left[\1_{\{u< U\}}(u-U)e^{-\theta v+\theta V}\right]+\E\left[\1_{\{u+a\wedge0\geq  U,  V-v>a\}}(v-V)e^{-\theta v+\theta V}\right]\\&\ \quad +\E\left[\1_{\{0\geq  U-u>a\wedge0,  V-v>a\}}(u-U)e^{-\theta v+\theta V}\right]
			\\&\ \quad+\E\left[Ue^{\theta (V-v)} \right]+ \E\left[\1_{\{ a\wedge 0+u\geq  U,V-v\geq U-u\}}\left( u-v-U+V \right) e^{\theta (V-v)} \right] \\&\
			={\color{orange}\E[Ue^{\theta V}]e^{-\theta v}} {\color{blue}+\E[(u-U)e^{\theta V}\mid U> u]\P(U> u)e^{-\theta v}}\\&\ \quad {\color{red} +e^{\theta a}\E[(u-U)K_0^Z(v+a+X)\1_{\{0\geq U-u>a\wedge 0, V>v+a\}}]}\\&\ \quad{\color{green}-e^{\theta a}\E [K_1^Z(v+a+X)\1_{\{Z> v+a+X, Y \leq  u+X+a\wedge 0\}}]}\\&\ \quad {\color{green}-ae^{\theta a}\E [K_0^Z(v+a+X)\1_{\{Z> v+a+X, Y \leq  u+X+a\wedge 0\}}]}\\&\quad\  {\color{purple}+\E\left[K_1^Z(v-u+Y)\1_{%u-v\leq
					\{Z>v-u+Y, Y\leq  u+ X+a\wedge 0\}}e^{\theta( Y-X-u)}\right]}
			%\\&\quad \ +\E\left[(K_1^Z(0)-K_0^Z(0)(Y-u+v))\1_{\{Y\leq  u+X\wedge (-v)\}}e^{-\theta(X+v)}\right]
			\intertext{$ \ $}
			F_D(u,v)&:= e^{-\theta u} -\E\left[\1_{\{ u\geq U, v+a\geq V\}}e^{-\theta \left((u- U)\wedge(v- V) \right)} \right]\\&\ =\E\left[\1_{\{ u< U \text{ or } v+a< V\}}e^{-\theta u+\theta U}\right] \\&\quad \ +\E\left[\1_{\{ u\geq U, v+a\geq V\}}\left\lbrace e^{-\theta u+\theta U}-e^{-\theta \left((u- U)\wedge(v- V) \right)} \right\rbrace \right]\\&\ =\E\left[\1_{\{ u< U\}}e^{-\theta u+\theta U}\right]+\E\left[\1_{\{ a\wedge 0+u\geq U, V>a+v\}}e^{-\theta u+\theta U}\right] \\&\quad \ +\E\left[\1_{\{ u\geq U> a\wedge 0+u, V>a+v\}}e^{-\theta u+\theta U}\right]\\&\quad \ +\E\left[\1_{\{a\wedge 0+ u\geq U, a\geq  V-v\geq U-u\}}\left\lbrace e^{-\theta u+\theta U}-e^{-\theta v+\theta V} \right\rbrace \right]\\&
			\ =\E\left[\1_{\{ u< U\}}e^{-\theta u+\theta U}\right]+\E\left[\1_{\{ a\wedge 0+u\geq U, V-v>U-u\}}e^{-\theta u+\theta U}\right] \\&\quad \ +\E\left[\1_{\{ u\geq U> a\wedge 0+u, V>a+v\}}e^{-\theta u+\theta U}\right]
\\&\quad \ -\E\left[\1_{\{a\wedge 0+ u\geq U,  V-v\geq U-u\}}e^{-\theta v+\theta V} \right]+\E\left[\1_{\{a\wedge 0+ u\geq U,  V-v>a\}}e^{-\theta v+\theta V} \right]
\end{align*}\begin{align*}&\ ={\color{violet}\E[K_0^Y(u+X)\1_{\{Y> u+X\geq 0\}}]}\\&\ \quad {\color{pink}+\E[K_0^Y(0)e^{-\theta (u+X)}\1_{\{Y>0> u+X\}}]}{\color{brown}+\E[e^{-\theta (u+X)}\1_{\{Y=0>u+X\}}]}\\&\ \quad {\color{purple}+\E\left[(1-K_0^Z(v-u+Y))\1_{%u-v\leq
				\{Z>v-u+Y, Y\leq  u+ X+a\wedge 0\}}e^{\theta( Y-X-u)}\right]}%\\&\ \quad +
	%	\E\left[	e^{-\theta u+\theta Y-\theta X}\1_{\{Y\leq u+X\wedge (-v)\}}\right]
		%	\\& \ \quad 	-\E\left[K_0^Z(0)	e^{-\theta (v+ X)}\1_{\{Y\leq u+X\wedge (-v)\}}\right]
			\\&\ \quad {\color{red} +\E\left[\1_{\{ u\geq U> a\wedge 0+u, V>a+v\}}e^{-\theta u+\theta U}\right]}\\&\ \quad{\color{green}+e^{\theta a}\E [K_0^Z(v+a+X)\1_{\{Z> v+a+X, Y \leq  u+X+a\wedge 0\}}]}
\end{align*}
Lastly
\begin{align*}
			F_1(u,v)&:=\P(U> u) +\P(V> v+a, U\leq  u)\\&\ ={\color{violet}\P(Y>u+X\geq 0)}{\color{pink} +\P(Y>0>u+X)}{\color{brown} +\P(Y=0>u+X)}\\&\ \quad {\color{green}+\P(Z> v+a+X, Y \leq  u+X+a\wedge 0)}\\&\ \quad {\color{red}+\P(Z> v+a+X, u+X+a\wedge 0<  Y \leq  u+X)}
		\end{align*}
\end{proof}

\section{Extension to Alternating Renewal Processes: $AR/G/1\rightarrow\cdot/G/1$}\label{app:AR}
\begin{theorem}\label{th:psigammaetaAR}
	Let the alternate renewal process $X_k$ defined in~\ref{sec:extensions} and $Y,Z$ be two r.v.'s (the service times). Define $U^{(i)}\simeq  Y-X^{(i)}$ and $V^{(i)}\simeq Z-X^{(i)}$ and assume that
 $\P(U^{(i)}>0)>0$ and $\P(V^{(i)}>0)>0$ for $i=1,2$.
	\begin{enumerate}[label=(\alph*)]
		\item The coupled integral equations $i=1,2$
		\begin{equation}\label{equationpsiAR}
			\E\left[\1_{\{u+X^{(i)}\geq  Y\}}\psi_{i}\left((u-U^{(i)})\wedge (v-V^{(i)}), v-V^{(i)} \right)\right]=\psi_{3-i}(u,v)
		\end{equation}
		admit a unique solution in the class of bounded functions on $\Rbar^2$  having the limits $\psi_i(\infty,\infty)=\lim_{u,v\to \infty}\psi_i(u,v)=1$. This is given by
		$$\psi_i(u,v):=\P( T^1_{1}\leq u, T^{2}_{1}\leq v \mid B=i)~,$$ where
		\[ T^1_k:=\sup_{k\leq  i<\infty}(Y_k-X_k)+\cdots+(Y_{i}-X_i)\]
		\[ T^2_k:=\sup_{k\leq  i<j<\infty}(Z_k-X_k)+\cdots+(Z_{i}-X_i)+(Y_{i+1}-X_{i+1})+\cdots+(Y_{j}-X_j)\]
		for $k\geq1$ and $(Y_1, Z_1)$, $(Y_2, Z_2),   \ldots $ being i.i.d. copies of $(Y, Z)$.
		\item	Assume that the functions $\gamma_i:\Rbar^2\to (-\infty,K_\gamma], i=1,2$, for some finite $K_\gamma$, satisfy
		for all  $(u,v)\in \supp(\gamma_1\vee \gamma_2\vee 0)\subseteq \Rbar^2$
		\begin{equation}\label{pre-estimatepsiAR}
		\E\left[\1_{\{u+X^{(i)}\geq  Y\}}\gamma_{i}\left((u-U^{(i)})\wedge (v-V^{(i)}), v-V^{(i)} \right)\right]\geq \gamma_{3-i}(u,v)~.
		\end{equation}
		If $\gamma_i(\infty, \infty):=\limsup_{u,v\to \infty}\gamma_i(u,v)=1$  then $\psi_i\geq\gamma_i$ for $i=1,2$.
		\item	Assume that the  functions $\eta_i:\Rbar^2\to [0,\infty), i=1,2$ satisfy
		for all  $(u,v)\in \supp(\psi_1\vee \psi_2)$ with $v\geq u$
		\begin{equation}\label{pre-estimateetaAR}
				\E\left[\1_{\{u+X^{(i)}\geq  Y\}}\eta_{i}\left((u-U^{(i)})\wedge (v-V^{(i)}), v-V^{(i)} \right)\right]\geq \eta_{3-i}(u,v)~.
		\end{equation}
		If
		$\eta_i(\infty, \infty):=\liminf_{u,v\to \infty}\eta_i(u,v)=1$ then $\psi_i(u,v)\leq\eta_i(u,v)$ for all $v\geq u$ and $i=1,2$.
	\end{enumerate}
\end{theorem}
The proof is (very) similar to that of Theorem~\ref{th:psigammaeta}.

\begin{proof}
	For part (a) we first prove that the given $\psi$ satisfies \eqref{equationpsiAR}; uniqueness will follow after proving (b).  We have for some $i\in\{1,2\}$
	\begin{align*}
		&\psi_{3-i}(u,v)=\P(T^{1}_{1}\leq u, T^{2}_{1}\leq v \mid B=3-i)
		\\&=\P\left( T^{1}_2\leq  (u+X_1-Y_1)\wedge (v+X_1-Z_1), T^{2}_{2}\leq v+X_1-Z_1   , Y_1\leq u+X_1\mid B=i\right),\\&=\P\left( T^{1}_1\leq (u+X^{(i)}-Y)\wedge (v+X^{(i)}-Z), T^{2}_{1}\leq v+X^{(i)}-Z  , Y\leq u+X^{(i)} \right),\intertext{where $(Y,Z)$ is independent of $(T^{1}_{1},T^{2}_{1})$. So}
		&=\E\left[\1_{\{u+X^{(i)}\geq  Y\}}\psi_{i}\left((u-U^{(i)})\wedge (v-V^{(i)}), v-V^{(i)} \right)\right].
	\end{align*}
	
	To prove part (b), i.e., $\psi_i\geq \gamma_i$, define first the functions
	\[f_i(u,v):=\limsup_{(x,y)\to (u,v)}\left(\gamma_i(x,y)-\psi_i(x,y)\right)~\forall (u,v)\in \Rbar^2~, i=1,2\]
	which are upper-semi continuous and attain their maximum on any closed subset of $\Rbar^2$. Let
	\[K:=\sup_{(u,v)\in \Rbar^2}(f_1(u,v)\vee f_2(u,v))~.\]
	If $K\leq 0$ the proof is complete; assume otherwise that $K>0$. Define
	\[\Kcal:=\{(u,v)\in \Rbar^2: f_1(u,v)\vee f_2(u,v)=K\}~, \]
	which is a closed subset of $\Rbar^2$, and
	\[a:=\min\{u\in \Rbar: \exists v\in \Rbar: (u,v)\in \Kcal \}\]	\[b:=\min\{v\in \Rbar: (a,v)\in \Kcal \}~.\]
	Since $f_{j}(a,b)=K>0$ for either $j=1$ or $j=2$, there exists a sequence $(a_n,b_n)\in  {\supp (\gamma_1\vee \gamma_2\vee 0)}$ such that $(a_n,b_n)\to (a,b)$ as $n\to \infty$
	\begin{align*}K&=f_{j}(a,b)=\lim_{n\to \infty}(\gamma_{3-j}-\psi_{3-j})(a_n,b_n)\\&\leq \limsup_{n\to \infty}\E\left[\1_{\{a_n\geq U^{(3-j)}\}}(\gamma_{3-j}-\psi_{3-j})\left((a_n-U^{(3-j)})\wedge (b_n-V^{(3-j)}), b_n-V^{(3-j)}\right)\right]~.\intertext{Since $\gamma_{3-j}-\psi_{3-j}\leq K_\gamma<\infty$, we can further use Fatou's lemma}&\leq \E\left[\limsup_{n\to \infty}\1_{\{a_n\geq U^{(3-j)}\}}(\gamma_{3-j}-\psi_{3-j})\left((a_n-U^{(3-j)})\wedge (b_n-V^{(3-j)}), b_n-V^{(3-j)}\right)\right]\\&\leq K\cdot\P(a=U^{(3-j)}) +\E\left[\1_{\{a>U^{(3-j)}\}}f_{3-j}\left((a-U^{(3-j)})\wedge (b-V^{(3-j)}), b-V^{(3-j)}\right)\right]\\&\leq K\cdot\P(a \geq U^{(3-j)})~,\end{align*}
	using the definitions of $K$ and $f_j$'s. It then follows that $\P(a\geq U^{(3-j)})=1$, such that necessarily
	\begin{equation} \label{abinfAR}
		f_{3-j}\left((a-U^{(3-j)})\wedge (b-V^{(3-j)}), b-V^{(3-j)}\right)=K~
	\end{equation}
	holds a.s. for the inequalities above to hold as equalities.
	
	Next we claim that $(a,b)=(\infty,\infty)$. Assume by contradiction that $a<\infty$. Then
	\eqref{abinfAR} and  $\P(U^{3-j}>  0)>0$ contradict with the choice of $a$, and hence $a=\infty$. Similarly, assume by contradiction that $b<\infty$. Then \eqref{abinfAR} and $\P(V^{3-j}>0)>0$ contradict with the choice of $b$, and hence $b=\infty$ as well.
	Finally, \[K=	f_{j}(\infty, \infty)=\limsup_{u,v\to \infty}(\gamma_{j}-\psi_{j})(u,v)=0\]
	from the limiting conditions on $\gamma_i$ and $\psi_i$, and hence $\psi_i\geq \gamma_i$.
	
	We can now prove the uniqueness of $\psi=(\psi_1,\psi_2)$ solving for (\ref{equationpsiAR}). Let $\psi^{(1)}$ and $\psi^{(2)}$ be two bounded solutions satisfying $\psi_j^{(i)}(\infty,\infty)=\lim_{u,v\to \infty}\psi_j^{(i)}(u,v)=1, i,j\in\{1,2\}$. Applying the second part of the theorem with $\psi=\psi^{(i)}$ and $\gamma=\psi^{(3-i)}$ (note that the proof only needs that $\psi$ satisfies (\ref{equationpsiAR}), is bounded, and $\psi(\infty,\infty)=\lim_{u,v\to \infty}\psi(u,v)=(1,1)$) we obtain that $\psi^{(i)}\geq \psi^{(3-i)}$ for $i=1,2$, and hence $\psi^{(1)}=\psi^{(2)}$.
	
	The proof for the last part of the theorem, i.e., $\psi(u,v)\leq \eta(u,v)$ on $v\geq u$ is similar and is omitted.
\end{proof}

The connection between the generic functions $\gamma$ and $\eta$ from Parts (b) and (c) of Theorem~\ref{th:psigammaetaAR}, and bounds on $\P(W>x)$, is also immediate:

\begin{corollary}{\sc{(Generic Upper and Lower Bounds)}}\label{cor:WgammaetaAR}
	Consider the functions $\psi_i$, $\gamma_i$, and $\eta_i$, for $i=1,2$, as in Theorem~\ref{th:psigammaetaAR}. Then the waiting time of a job $n\to\infty$ satisfies for all $x\geq 0$
	\begin{align*}
		&1-\frac{1}{2}\E\left[(\eta_1+\eta_2)(x-(Z-Y)_+, x-(Z-Y) )\right]\\&\quad\leq \P(W>x)=1-\frac{1}{2}\E\left[(\psi_1+\psi_2)(x-(Z-Y)_+, x-(Z-Y) )\right]\notag\\&\quad\leq 1-\frac{1}{2}\E\left[(\gamma_1+\gamma_2)(x-(Z-Y)_+, x-(Z-Y))\right]~.
	\end{align*}
	The corresponding sojourn time satisfies
	\begin{align*}
		&1-\frac{1}{2}\E\left[\1_{\{Z_1+Y<x\}}(\eta_1+\eta_2)(x-(Z_1+Z_2\vee Y),x-(Z_1+Z_2)) \right]\notag\\&
		\qquad\leq \P(D>x)=1-\frac{1}{2}\E\left[\1_{\{Z_1+Y<x\}}(\psi_1+\psi_2)(x-(Z_1+Z_2\vee Y),x-(Z_1+Z_2)) \right]\notag\\&\qquad\leq 1-\frac{1}{2}\E\left[\1_{\{Z_1+Y<x\}}(\gamma_1+\gamma_2)(x-(Z_1+Z_2\vee Y),x-(Z_1+Z_2)) \right]~.
	\end{align*}
\end{corollary}

The proof is a forward extension of the proof of Corollary~\ref{cor:Wgammaeta}.

\begin{proof}
	From $W$'s representation from (\ref{eq:W2nodes}) it follows for all $x\geq0$
	\[\begin{aligned}\P(W>x)&=\P(\max\left\{0,T^1+(Z-Y)_+, T^2+Z-Y\right\}>x)\\&=1-\P(\max\left\{0,T^1+(Z-Y)_+, T^2+Z-Y\right\}\leq x)\\&=1-\P\left(T^1\leq x- (Z-Y)_+, T^2\leq x- (Z-Y)\right)\\&=1-\sum_{i=1}^2\P\left(T^1\leq x- (Z-Y)_+, T^2\leq x- (Z-Y)\mid B=i\right) \P(B=i)\\&=1-\frac{1}{2}\E\left[(\psi_1+\psi_2)(x-(Z-Y)_+, x-(Z-Y))\right]~.\end{aligned}\]
	Since $x-(Z-Y)_+\leq  x-(Z-Y)$ and $\gamma(u,v)\leq\psi(u,v)\leq \eta(u,v)~\forall v\geq u$, the upper and lower bounds on $P(W>x)$ follow immediately. The bounds on $\P(D>x)$ follow from $D=W+Z_1+Y$.
\end{proof}

\section{Extension to Tandem Packet Networks}\label{app:Kleinrock}

		\begin{theorem}\label{th:psigammaetaKleinrock}
			Let $Y\geq 0$ and $X\geq 0$ be  two random variables  satisfying $\P(Y> 0)>0$. Let $(U,V)=(Y-X, MY-X)$
			\begin{itemize}
				\item The integral equation
				\begin{equation}\label{equationpsiKlein}
					\E\left[\1_{\{U\leq u, V\leq v \}}\psi((u-U)\wedge(v-V), v-U)\right]=\psi(u,v)
				\end{equation}
				admits a unique solution in the class of bounded functions in $\Rbar^2$ having the limit $\psi(\infty,\infty)=\lim_{u,v\to \infty}\psi(u,v)=1$, i.e.,
				\[\psi(u,v):=\P(T^1_1<u, T^2_1<v)~,\]  where
				\[T^1_{k}=\max_{k\leq i<\infty}((Y_k-X_k)+\ldots +(Y_i-X_i))\]
			\[T^2_{k}=\max_{k\leq i<\infty}((Y_k-X_k)+\ldots +(Y_i-X_i)+(M-1)\max(Y_k,\ldots , Y_i))\]
				for $k\geq1$ and $(X_1, Y_1)$, $(X_2, Y_2),   \ldots $ being i.i.d. copies of $(X, Y)$.
				\item	Assume that the function $\gamma:\Rbar^2\to (-\infty,K_\gamma]$, for some finite $K_\gamma$, satisfies
				for all  $(u,v)\in \supp(\gamma\vee 0)\subseteq \Rbar^2$
				\begin{equation}\label{pre-estimatepsiKlein}
					\E\left[\1_{\{U\leq u, V\leq v \}}\gamma((u-U)\wedge(v-V), v-U)\right]\geq \gamma(u,v)~.
				\end{equation}
				If $\gamma(\infty, \infty):=\limsup_{u,v\to \infty}\gamma(u,v)=1$  then $\psi\geq\gamma$.
				\item	Assume that the  function $\eta:\Rbar^2\to [0,\infty)$ satisfies
				for all  $(u,v)\in \supp(\psi)$ with $v\geq u$
				\begin{equation}\label{pre-estimateetaKlein}
							\E\left[\1_{\{U\leq u, V\leq v \}}\eta((u-U)\wedge(v-V), v-U)\right]\leq \eta(u,v)~.
				\end{equation}
				If	$\eta(\infty, \infty):=\liminf_{u,v\to \infty}\eta(u,v)=1$ then $\psi(u,v)\leq\eta(u,v)$ for all $v\geq u$.
			\end{itemize}
		\end{theorem}

An interesting aspect about the theorem is that the second component in the functions $\psi$, $\gamma$, and $\eta$ is $v-U$ as opposed to $v-V$, as in Theorem~\ref{th:psigammaeta}; the proof is (very) similar to that of Theorem~\ref{th:psigammaeta}:
		
		\begin{proof}
			For part (a) we first prove that the given $\psi$ satisfies \eqref{equationpsiKlein}; uniqueness will follow after proving (b).  We have
			\begin{align*}
				&\psi(u,v)=\P(T^{1}_{1}\leq u, T^{2}_{1}\leq v  )\\&=\P\left( Y_1-X_1+0\vee T^{1}_2\leq  u,   MY_1-X_1+0\vee T^{1}_{2}\leq v , Y_1-X_1+T^{2}_{2}\leq v\right)
				\\&=\P\big( Y_1-X_1\leq u,MY_1-X_1\leq v,\\&\qquad \quad  T^{1}_2\leq  (u+X_1-Y_1)\wedge (v+X_1-MY_1), T^{2}_{2}\leq v+X_1-Y_1\big)\\&=	\E\left[\1_{\{Y-X\leq u, MY-X\leq v \}}\psi((u+X-Y)\wedge(v+X-MY), v+X-Y)\right]\\&=	\E\left[\1_{\{U\leq u, V\leq v \}}\psi((u-U)\wedge(v-V), v-U)\right].
			\end{align*}
			
			To prove part (b), i.e., $\psi\geq \gamma$, define first the function
			\[f(u,v):=\limsup_{(x,y)\to (u,v)}\left(\gamma(x,y)-\psi(x,y)\right)~\forall (u,v)\in \Rbar^2~,\]
			which is upper-semi continuous and attains its maximum on any closed subset of $\Rbar^2$. Let
			\[K:=\sup_{(u,v)\in \Rbar^2}f(u,v)~.\]
			If $K\leq 0$ the proof is complete; assume otherwise that $K>0$. Define
			\[\Kcal:=\{(u,v)\in \Rbar^2: f(u,v)=K\}~, \]
			which is a closed subset of $\Rbar^2$, and
			\[a:=\min\{u\in \Rbar: \exists v\in \Rbar: (u,v)\in \Kcal \}\]
			\[b:=\min\{v\in \Rbar: (a,v)\in \Kcal \}~.\]
			Since $f(a,b)=K>0$, there exists a sequence $(a_n,b_n)\in  {\supp (\gamma\vee 0)}$ such that $(a_n,b_n)\to (a,b)$ as $n\to \infty$ and
			\begin{align*}K&=f(a,b)=\lim_{n\to \infty}(\gamma-\psi)(a_n,b_n)\\&\leq \limsup_{n\to \infty}\E\left[\1_{\{a_n\geq  U, b_n\geq V\}}(\gamma-\psi)\left((a_n-U)\wedge (b_n-V), b_n-U\right)\right]\intertext{Since $\gamma-\psi\leq K_\gamma<\infty$, we can further use Fatou's lemma}&\leq \E\left[\limsup_{n\to \infty}\1_{\{a_n\geq  U, b_n\geq V\}}(\gamma-\psi)\left((a_n-U)\wedge (b_n-V), b_n-U\right)\right]\\&\leq K\cdot\P(a=U \text{ or } b=V ) +\E\left[\1_{\{a> U, b>V\}}f\left((a-U)\wedge (b-V), b-V\right)\right]\\&\leq K\cdot\P(a\geq  U, b\geq V)~,\end{align*}
			using the definitions of $K$ and $f$. It then follows that $\P(a\geq U, b\geq V)=1$, such that necessarily
			\begin{equation} \label{abinfKlein}
				f\left((a-U)\wedge (b-V), b-U\right)=K~
			\end{equation}
			holds a.s., for the inequalities above to hold as equalities.
			
			Next we claim that $(a,b)=(\infty,\infty)$. Assume by contradiction that $a<\infty$. Then
			\eqref{abinfKlein} and  $\P(U>  0)>0$ contradict with the choice of $a$, and hence $a=\infty$. Similarly, assume by contradiction that $b<\infty$. Then \eqref{abinfKlein} and $\P(U>0)>0$ contradict with the choice of $b$, and hence $b=\infty$ as well.
			Finally, \[K=	f(\infty, \infty)=\limsup_{u,v\to \infty}(\gamma-\psi)(u,v)=0\]
			from the limiting conditions on $\gamma$ and $\psi$, and hence $\psi\geq \gamma$.
			
			We can now prove the uniqueness of $\psi$ solving for (\ref{equationpsiKlein}). Let $\psi_1$ and $\psi_2$ be two bounded solutions satisfying $\psi_i(\infty,\infty)=\lim_{u,v \to \infty}\psi_i(u,v)=1$. Applying the second part of the theorem with $\psi=\psi_i$ and $\gamma=\psi_{3-i}$ (note that the proof only needs that $\psi$ satisfies (\ref{equationpsiKlein}), is bounded, and $\psi(\infty,\infty)=\lim_{u,v \to \infty}\psi(u,v)=1$) we obtain that $\psi_i\geq \psi_{3-i}$ for $i=1,2$, and hence $\psi_1=\psi_2$.
			The proof for the last part of the theorem, i.e., $\psi(u,v)\leq \eta(u,v)$ on $v\geq u$ is also similar with the corresponding one from Theorem~\ref{th:psigammaeta}, and is omitted here.
		\end{proof}

\begin{corollary}{\sc{(Generic Upper and Lower Bounds)}}\label{cor:WgammaetaKlein}
		Consider the functions $\psi$, $\gamma$, and $\eta$ as in Theorem~\ref{th:psigammaetaKleinrock}. Then the waiting time of a job $n\to\infty$ satisfies for all $x\geq 0$
		\begin{align*}
			1-\E\left[\eta(x, x+(M-1)Y)\right]&\leq \P(W>x)=1-\E\left[\psi(x, x+(M-1)Y)\right]\notag\\&\leq 1-\E\left[\gamma(x, x+(M-1)Y)\right]~.
		\end{align*}
		The corresponding sojourn time satisfies
		\begin{align*}
		1-\E\left[\1_{\{MY\leq x\}}\eta(x-MY, x-Y)\right]&\leq \P(D>x)=1-\E\left[\1_{\{MY\leq x\}}\varphi(x-MY, x-Y)\right]\notag\\&\leq 1-\E\left[\1_{\{MY\leq x\}}\psi(x-MY, x-Y)\right]~.
		\end{align*}
	\end{corollary}
\begin{proof}
	From $W$'s representation from \eqref{eq:WKleinrock} it follows for all $x\geq0$
	\[\begin{aligned}\P(W>x)&=\P(\max\left\{0,T^1, T^2+(M-1)Y\right\}>x)\\&=1-\P(\max\left\{0,T^1, T^2+(M-1)Y\right\}\leq x)\\&=1-\P\left(T^1\leq x, T^2\leq x+(M-1)Y\right) \\&=1-\E\left[\psi(x, x+(M-1)Y)\right]~.\end{aligned}\]
	Since $x\leq  x+(M-1)Y$ and $\gamma(u,v)\leq\psi(u,v)\leq \eta(u,v)~\forall v\geq u$, the upper and lower bounds on $P(W>x)$ follow immediately. The bounds on $\P(D>x)$ follow from $D=W+MY$.
\end{proof}

\end{document}